\documentclass[reqno,12pt]{amsart}
\usepackage{amsmath,amsthm,amsfonts,amssymb,amscd, xcolor}

\usepackage{geometry}

\theoremstyle{plain}
\newtheorem{maintheorem}{Theorem}

\newtheorem{maincorollary}[maintheorem]{Corollary}

\newtheorem{theorem}{Theorem}
\newtheorem*{theorem*}{Theorem}

\newtheorem{claim}{Claim}[section]
\newtheorem{lemma}[claim]{Lemma}
\newtheorem{remark}[claim]{Remark}

\newtheorem{definition}[claim]{Definition}
\newtheorem{proposition}[claim]{Proposition}

\newcommand{\D}{\ensuremath{\mathbb{D}}}

\newcommand{\N}{\ensuremath{\mathbb{N}}}

\newcommand{\R}{\ensuremath{\mathbb{R}}}

\newcommand{\T}{\ensuremath{\mathbb{T}}}

\newcommand{\Z}{\ensuremath{\mathbb{Z}}}

\newcommand{\cC}{{\mathcal C}}

\newcommand{\cF}{{\mathcal F}}
\newcommand{\Fc}{{\mathcal F}}

\newcommand{\cL}{{\mathcal L}}

\newcommand{\cN}{{\mathcal N}}

\newcommand{\cP}{{\mathcal P}}

\newcommand{\cR}{{\mathcal R}}
\newcommand{\cS}{{\mathcal S}}

\newcommand{\supp}{\operatorname{supp}}
\newcommand{\ds}{\displaystyle}

\newcommand{\ba}{{\textbf{a}}}
\newcommand{\bb}{{\textbf{b}}}
\newcommand{\bc}{{\textbf{c}}}

\newcommand{\bt}{{\textbf{t}}}
\newcommand{\sC}{C}
\newcommand{\sE}{E}

\newcommand{\torus}{{\mathbb{T}^u}}
\newcommand{\Int}{\mbox{Int}\,}

\newcommand{\fm}{\mathfrak{m}}

\begin{document}
\title[Regularity of the density of SRB measures]{Regularity of the density of SRB measures for solenoidal attractors}
	
\date{}
\author{Carlos Bocker}
\address[Carlos Bocker]{Department of Mathematics, UFPB\\ Jo\~ao Pessoa-PB, Brazil}
\email{cbocker@gmail.com}

\author{Ricardo Bortolotti}
\address[Ricardo Bortolotti]{Department of Mathematics, UFPE\\ Recife-PE, Brazil}
\email{ricardo@dmat.ufpe.br}

\begin{abstract}
We show that a 
 class of higher-dimensional hyperbolic endomorphisms admit absolutely continuous invariant probabilities whose density are regular.
The maps we consider are given by 
$ T(x,y) = (\sE (x), \sC(y) + f(x) )$,
 where $\sE$ is a linear expanding map of $\torus$, $\sC$ is a linear contracting map of $\R^d$, $f$ is in $C^r(\torus,\R^d)$ and $r \geq 2$.
We prove that if 
$|(\det \sC)(\det\sE)| \|\sC^{-1}\|^{-2s}>1$ for some $s<r-(\frac{u+d}{2}+1)$ 
 and $T$ satisfies a certain transversality condition,
then the density of the SRB measure of $T$ is contained in the Sobolev space  $H^s(\T^u\times \R^d)$
, in particular, if
$s>\frac{u+d}{2}$ then the density is $C^k$ for every $k<s-\frac{u+d}{2}$.
We also exhibit a condition involving $\sE$ and $\sC$ under which this tranversality condition is valid for almost every $f$.
\end{abstract}

\maketitle

\section{Introduction}

The ergodic theory of hyperbolic endomorphisms was developed in the last years and presents similar results to the ergodic theory of invertible hyperbolic dynamics such as SRB measures, equilibrium states and structural stability \cite{jiang, przytycki, qian.xie,  qian.zhang, urbansky.wolf}.

One interesting phenomena that may occur for hyperbolic endomorphisms is that the SRB measure needs not to be singular when the dynamic expands volume, what does not happens for hyperbolic proper attractors \cite{alves.pinheiro, bowen.ruelle}. This was  observed in \cite{fat.solenoidal, tsujii.acta} and extended in \cite{bocker.bortolotti}, where was proved the absolute continuity of the SRB measure under certain geometrical transversality condition.

The absolute continuity of the SRB measure is usually associated to maps with only positive Lyapunov exponents \cite{adler, alves}. The main feature in \cite{avila.gouezel.tsujii, bocker.bortolotti, fat.solenoidal} is a geometrical condition of transversality between the images of the unstable directions that allows to conclude properties of regularity of the SRB measure that  are similar to those that occurs for expanding maps. Due to these results, one may expect that volume expanding hyperbolic attractors for endomorphisms satisfies ergodic properties similar to expanding maps.

Since the density of the SRB measure is smooth for expanding maps \cite{expanding.dif1, expanding.dif2}, one should ask whether the property of the smoothness of the density is also valid for volume expanding hyperbolic endomorphisms under this transversality condition. Here we prove the Sobolev regularity of the density of the SRB measure.

We study the action of the operator $\cL$ on an appropriate Banach space $\mathcal{B}$ adapted to the dynamic. This Banach space is defined using the method developed in \cite{gouezel.liverani} that was also used in \cite{avila.gouezel.tsujii}, defining an anisotropic norm of the function corresponding to its action in the space of regular functions supported in ``almost stable manifolds' (see definition \ref{norm.dagger}). 
In this work, we consider maps  $T:\torus\times \R^d \to \torus\times \R^d $  given by

\begin{equation}
T(x,y) = ( E(x), C(y) + f(x) ),
\end{equation}
where $E$ is a linear expanding map of the torus $\torus$, $C$ is a linear contraction of $\R^d$, $f\in C^r(\torus,\R^d)$ and $r\geq 2$.

In \cite{bocker.bortolotti} the authors gave sufficient conditions for the absolute continuity of the SRB measure $\mu_T$ of $T$. 
In this paper, we study the Sobolev regularity of the density $\phi_T = d\mu_T/dx$. The low-dimensional case $u=d=1$ was previously studied in \cite{avila.gouezel.tsujii}. Here we are focused on the higher-dimensional setting with $u\geq d \geq 1$.

Denote by $\mathcal{E}(u)$ the set of the linear expanding maps of $\T^u$, by $\mathcal{C}(d)$ the set of the linear contractions of $\R^d$ and denote $T=T(E,C,f)$ for $E\in \mathcal{E}(u)$, $C\in \mathcal{C}(d)$ and $f\in C^r(\torus,\R^d)$.
Given $\sE\in \mathcal{E}(u)$, consider the following subset $\mathcal{C}_s(d;E)$ of $\mathcal{C}(d)$:
$$ \mathcal{C}_s(d;\sE)=\left\{ \sC \in \mathcal{C}(d) , |\det \sC| |\det \sE| \|C^{-1}\|^{-2s} > 1 \text{ and } \|\sC\| < \frac{\|\sE^{-1}\|^{-1}}{|\det \sE|^{\frac{1}{{u-d+1}}} } \right\}.$$

When $T$ contracts volume $(|\det E| |\det C|<1)$ there exists no absolutely continuous invariant probability (ACIP). On the other hand, if $T$ expands volume then
the condition
$|\det \sC| |\det \sE| \|C^{-1}\|^{-2s} > 1$ is valid  for some $s>0$.

\begin{maintheorem}\label{teo.a}
Given integers $u \geq  d $ and $0 \leq s < r-(\frac{u+d}{2}+1)$, $\sE\in \mathcal{E}(u)$ and $C\in \mathcal{C}_s(d;E)$, there exists an open and dense subset $\mathcal{U}$ of $C^r(\torus, \R^d)$ such that the corresponding SRB measure $\mu_{T}$ of  $T=T(\sE,\sC,f)$ for $f\in\mathcal{U}$ is absolutely continuous with respect to the volume of $\torus\times\R^d$ and its density is in $H^s(\T^u\times \R^d)$.
\end{maintheorem}

The condition behind the subset $\mathcal{U}$ corresponds to a geometrical condition of transversal overlaps of the images (see definition \ref{transversality.1}). In \cite{bocker.bortolotti}, the authors proved that this condition is generic when $C$ is in $C_0(d;\sE)$.

Notice that if  $C\in \mathcal{C}_0(d;E)$ we obtain the absolute continuity of $\mu_T$ under the condition $|\det C| |\det E|>1$, which is more general than the hypothesis of \cite[Theorem A]{bocker.bortolotti}. Moreover, by continuity, if $C \in C_0(d;E)$ then $C \in \mathcal{C}_s(d;E)$ for some $s>0$.

\begin{maincorollary}\label{cor.b}
Given integers $u \geq  d $,  $\sE\in \mathcal{E}(u)$ and $C\in C_0(d;E)$, there exists an open and dense subset $\mathcal{U}$ of $C^r(\torus, \R^d)$ such that the corresponding SRB measure $\mu_{T}$ of every map $T=T(\sE,\sC,f)$ for $f\in\mathcal{U}$ is absolutely continuous with respect to the volume of $\torus\times\R^d$ and its density is in $H^s(\T^u\times \R^d)$ for some $s>0$.
\end{maincorollary}

In the situation where $s>\frac{u+d}{2}$, Sobolev's embedding theorem implies that any $\phi_T$ coincides almost everywhere with a $C^k$ function for every $k<s-\frac{u+d}{2}$. In particular $\phi_T$ is continuous almost everywhere, that implies that the attractor $\Lambda$ has non-empty interior.

\begin{maincorollary}\label{cor.c}
	Under the assumptions of Theorem \ref{teo.a}, if $r\geq u+d+2$ and $s>\frac{u+d}{2}$, then the corresponding attractor $\Lambda_T$ has non-empty interior.
\end{maincorollary}

Consider the Ruelle-Perron-Frobenius transfer operator (or simply transfer operator) $\cL:L^1 \to L^1$ defined by
\begin{equation}\label{eq. perron}
\cL \phi(x)= \sum_{T(y)=x}\frac{\phi(y)}{|\det DT(y)|}.
\end{equation}

The technical part of this paper corresponds a Lasota-Yorke inequality for the transfer operator in a Banach space $\mathcal{B}$ contained in $H^s$. This kind of approach also allows to conclude statistical properties as consequences of the spectral gap. Actually, for $s>u/2$ we prove the existence of a spectral gap for $\cL$ and, thus, exponential decay of correlations for $T$ in a Banach space containing smooth observables.

\begin{maintheorem}\label{teo.d} Suppose that $\sC \in \mathcal{C}_s(d;E) $ for $u/2 < s< r-(\frac{u+d}{2}+1)$ and 
	$$
	\zeta \in (\max\{\|E^{-1}\|^{\frac{1}{1+\log{(1+\frac{d}{2}+u)}}},(|\det E||\det C|\|C^{-1}\|^{2s})^{\frac{1}{2}}\},1).
	$$ 

Then, for any $f \in C^r(\torus,\R^d)$ in an open and dense set,  there exists a Banach space $\mathcal{B}$ contained in $H^s(\torus\times \R^d)$ and containing $C^{r-1}(D)$ such that the action of the operator $\cL$ in $\mathcal{B}$ has spectral gap with essential spectral radius at most $\zeta$. In particular, $T$ has exponential decay of correlations in some linear space $\mathcal{\tilde B}$ with exponential rate $\zeta$, where  $\mathcal{\tilde B}$ is contained in $\mathcal{B}$ and contains $C^{r-1}(D)$ .
\end{maintheorem}	

An interesting consequence of Theorem D is that the rate of exponential decay of correlations can be taken uniform when the rate of contraction tends to be weaker, for instance, through the family of dynamics $T_t=T(E,(1-t)C+tI,f)$, $0\leq t <1$. An interesting problem is to show that $T_1$ has exponential decay of correlations with the same rate $\zeta$ for an open an dense set of $f$'s.

The plan of the paper is as follows: Section 2 details the basic definitions (including the transversality condition) and statements of this work. In section 3 we introduce the norms, some properties that shall be used further and two Main Lasota-Yorke inequalities for the transfer operator. Section 4 is dedicated to the proof of the two Main Lasota-Yorke inequalities.
In section 5 we prove a third Lasota-Yorke Inequality and we prove Theorems 1 and 2. In Section 6 we conclude Theorems A, D and Corollaries B, C as consequence of the genericity of the transversality condition when $C\in \mathcal{C}(d;E)$.

\section{Definitions and statements}

Given integers $u$ and $d$, we consider the
dynamic $T=T(\sE,\sC,f):\torus\times \R^{d} \rightarrow \torus\times \R^{d}$ given by
\begin{equation}
T(x,y) = \big ( \sE(x) , \sC(y)  + f(x)  \big)\,,
\end{equation}
where $\sE\in \mathcal{E}(u)$ is a map whose lift $E:\R^u \rightarrow \R^u$ is a linear map with $\|E^{-1}\|^{-1}>1$ that preserves the lattice $\Z^u$, $\sC\in \mathcal{C}(d)$ is a linear invertible map with $\|\sC\|<1$  and $f \in C^r(\torus,\R^d)$, $r \geq 2$.

The attractor $\Lambda$ for $T$ is given by $\Lambda=\cap_{n\geq 0}{T^n(D)}$ for some $D= \torus \times [-K_0,K_0]^d$  satisfying $T(D) \subset D$. Since the  restriction of $T$ to $\Lambda$ is a transitive hyperbolic endomorphism, it admits a unique SRB measure $\mu_T$ supported on $\Lambda$ \cite{urbansky.wolf}.

We suppose in the whole text that $T$ is volume expanding and we consider $s>0$ such that $|\det E| |\det C| \|C^{-1}\|^{-2s}>1$.

\subsection{Codifying the dynamics}

Let us fix notation involving the partition of the base space $\torus$ that codify the action of the expanding map $\sE$. This is essentially the same notation used in \cite{bocker.bortolotti}.

Fix $\mathcal{R}=\{\mathcal{R}(1),\cdots,\mathcal{R}(r)\}$  a Markov partition for $\sE$, that is,
$\mathcal{R}(i)$ are disjoint open sets, the interior of each $\overline{R(i)}$ coincides with $R(i)$,
$\sE_{|_{\mathcal{R}(i)}}$ is one-to-one,
${\bigcup}_i \overline{\mathcal{R}(i)}= \torus$  and
$\sE({\mathcal{R}(i)}) \cap {\mathcal{R}(j)} \neq \emptyset$ implies that ${\mathcal{R}(j)}\subset E(\mathcal{R}(i))$. Each $\mathcal{R}(i)$ is called a rectangle of the Markov partition.
Markov partitions always exist for expanding maps with arbitrarily small diameter (see \cite{markov.partition}).

Let us suppose that $\operatorname{diam}(\mathcal{R}) < \gamma$, where ${0<\gamma<1/2}$ is a constant such that for every $x\in \torus$ and $y\in \sE^{-1}(x)$ there exists a unique affine inverse branch
$g_{y,x}:B(x,\gamma) \to B(y,\gamma)$  such that
\begin{equation}\label{eq. ramo inverso}
g_{y,x}(x)=y \quad \text{and}\quad \sE(g_{y,x}(z))=z
\end{equation}
for every $z \in B(x,\gamma)$.

Consider the set $\overline{I}=\{ 1, \cdots, r\}$ and $\overline{I}^n$ the set of words of length $n$ with letters in $\overline{I}$, $1\leq n \leq \infty$. Denoting by $\ba =(a_{i})_{i=1}^{n}$ a word in $\overline{I}^{n}$, define $I^n$ the subset of admissible words $\ba =(a_{i})_{i=1}^{n}$, that is,  with the property that
\begin{equation}
\sE(\mathcal{R}(a_{i+1})) \cap \mathcal{R}(a_{i}) \neq \emptyset \text{ for every } 0\leq i \leq n-1 \,.
\end{equation}

Consider the partition $\mathcal{R}^{n}:= \vee_{i=0}^{n-1} \sE^{-i}(\mathcal{R})$ and, for every $\ba \in \overline{I}^n$,  the set $\mathcal{R}(\ba) = \cap_{i=0}^{n-1} \sE^{-i}(\mathcal{R}(a_{n-i}))$ in $ \mathcal{R}^n$, which is nonempty if and only if $\ba\in I^n$. The~truncation of $\ba=(a_j)_{j=1}^{n}$ to length $1\le p\le n$  is denoted by $[\ba]_p=(a_j)_{j=1}^{p}$.

For any $x\in\torus$, fix some $\pi(x) \in \overline{I}$ such that $x\in \overline{\mathcal{R}(\pi(x))}$.
For any $\bc\in I^p$, $1\leq p < \infty$, we consider $I^n(\bc)$ the set of words $\ba\in I^n$ such that $\sE^n(\mathcal{R}(\ba)) \cap \mathcal{R}(\bc) \neq \emptyset$.
Define $I^n(x):=I^n(\pi(x))$ and,
for $\ba\in I^{n}(x)$, denote by $\ba(x)$ the point $y\in\mathcal{R}(\ba)$ that satisfies $\sE^{n}(y)=x$.

For any $\ba\in I^n$ and $1\leq n <\infty$ we consider the set $\D(\ba):=\{x\in\torus | \ba \in I^n(x)\}=\sE^n(\cR(\ba))=\sE(\cR([\ba]_1))$, which is a union of rectangles of the Markov partition. The image of  $\mathcal{R}(\ba)\times\{ 0 \}$ by $T^{n}$ is the graph of the function $S(\cdot,\ba):\D(\ba) \to \R^d$ given by
\begin{equation}
S(x,\ba) := \sum_{i=1}^{n} \sC^{i-1} f (\sE^{n-i}(\ba(x))) = \sum_{i=1}^{n} \sC^{i-1} f([\ba]_{i} (x)).
\end{equation}

Consider the sets $I^\infty(x) = \{ \ba \in I^\infty$ such that $[\ba]_i \in I^i(x)$ for every $i\geq 1 \}$  and
$\D(\ba):=\{ x \in \torus | \ba\in I^\infty(x)\}=\cap_{n=1}^{+\infty}\sE^n(\cR([\ba]_n)) =\sE(\cR([\ba]_1))$ for $\ba\in I^\infty$.
If $\ba\in I^{\infty}(x)$, we define $S(x,\ba) = {\lim}_{n\to\infty} S(x,[\ba]_n)$.

For any $p\geq 1$ and $\bc\in I^p$,
let us denote by $\mathcal{R}_*(\bc)$ the union of atoms $\mathcal{R}(\tilde{\bc})$, $\tilde{\bc}\in I^p$,
that are adjacent to $\mathcal{R}(\bc)$.
We suppose that the diameter of the partition $\mathcal{R}$ is small enough such that the diameter of $\mathcal{R}_*(\bc)$ is smaller than $\gamma$. For $\ba \in I^i$, let us denote by $E^{-i}_{\bc,\ba}$ the inverse branch of $E^i$ satisfying $E^{-i}_{\bc,\ba}(\mathcal{R}(\bc))\subset \mathcal{R}(\ba)$ (and so $E^{-i}_{\bc,\ba}(\mathcal{R}_{*}(\bc))\subset \mathcal{R}_{*}(\ba)$).    We can extend $S(x,\ba)$ to a ball $B_\bc$ of radius $\gamma$ containing  $\mathcal{R}_*(\bc)$ by
\begin{equation}
S_\bc(x,\ba) := \sum_{i=1}^{n} \sC^{i-1} f(E^{-i}_{\bc,\ba}(x)).
\end{equation}

Consider the constant $\alpha_0 := \frac{\|f\|_{C^r}}{1-\|C\|} $. Notice also that
$S_\bc ( \cdot, \ba)$ is of class $C^r$ and
\begin{equation}\label{alpha_0}
\|E\|^j \|\partial^\alpha S_\bc(x,\ba)\| \leq \alpha_0
\end{equation}
for every $x \in \mathcal{R}_*(\bc)$ and multi-index $|\alpha| = j$, $0 \leq j \leq r$.

%

\subsection{The transversality condition}
	
Given a linear map $A:\R^{u}\rightarrow \R^{d}$, denote by
\begin{equation}
\displaystyle \fm(A):= \sup_{\dim W = d} \inf_{\|v\|=1, v\in W} {\|A(v)\|}
\end{equation}
the smallest singular value of $A$. Denote the minimum and maximum rates of expansion and contraction by $\underline{\mu}= \|\sE^{-1}\|^{-1}$, $\overline{\mu}= \|\sE\|$, $\underline{\lambda}= {\|\sC^{-1}\|^{-1}}$, $\overline{\lambda}= \|\sC\|$. Consider also $N=|\det \sE|$ the degree of the expanding map and  $\theta=\overline{\lambda}\underline{\mu}^{-1}$.

\begin{definition}\label{transversality.1}
Given $T=T(\sE,\sC,f)$ as above, integers $1\leq p, q <\infty$, $\bc\in I^p$ and $\ba, \bb \in I^{q}(\bc)$, we say that $\ba$ and $\bb$  are \textbf{transversal} on $\bc$ if
\begin{equation}\label{transversality}
\fm(   D S_{\bc}(x,\ba) - D S_{\bc}(y,\bb)    )   > 3 \theta^q \alpha_{0}
\end{equation}
for every $x, y \in \overline{\mathcal{R}_*(\bc)}$.
Defining the integer $\tau(q)$ by
\begin{equation}
\tau(q) =  \min_{p\geq 1} \max_{\bc\in I^{p}} \max_{\ba\in I^{q}(\bc)} \#\{ \bb\in I^{q}(\bc) | \text{ } \ba \text{ is not transversal to } \bb \text{ on } \bc \},
\end{equation}
we say that it holds the \textbf{transversality condition} if
 \begin{equation}\label{transversality.condition}
 \limsup_{q\to\infty}\frac{1}{q}\log \tau(q) =0.
 \end{equation}
\end{definition}

When $E$ and $C$ are fixed, we denote $\tau_f(q)$ to denote its dependence on $f$. In \cite{bocker.bortolotti}, it was given a  condition which implies that, for every $\beta>0$, the set of $f$'s satisfying
 	$\limsup_{q\to\infty}\frac{\log \tau(q)}{q}>\beta$ is open and dense. 
 	More precisely, considering
 		$$
 		\displaystyle \mathcal{C}(d;\sE)=\left\{ \sC \in \mathcal{C}(d) , \|\sC\| < \frac{\|\sE^{-1}\|^{-1}}{|\det \sE|^{\frac{1}{{u-d+1}}} } \right\},
 		$$
 		it was proved that there exists a residual subset $\mathcal{R}\subset C^r(\torus, \R^d)$ such that if $\sC\in  \mathcal{C}(d;\sE)$, then $\limsup \frac{\log \tau_f(q)}{q}=0$ for every $f \in \mathcal{R}$ (see Proposition \ref{bb}).
 	
Theorems A and D are obtaining putting together their more explicit formulations evolving the transversality condition given below with the genericity of the transversality condition.

\begin{theorem}\label{teo.1} Given $0\leq s < r-(\frac{u+d}{2}+1)$, $E \in \mathcal{E}(u)$, $C \in \mathcal{C}(d)$ and $f \in C^r(\torus,\R^d)$ such that $|\det E| |\det C| \|C^{-1}\|^{-2s} >1$ and the transversality condition is valid, then there exists an open set $\mathcal{U} \subset \mathcal{C}(d) \times C^r(\torus, \R^d)$ containing $f$ such that for every $(\tilde{C},\tilde{f})\in\mathcal{U}$ the SRB measure $\mu_T$ for $T=T(E,C,\tilde{f})$ is absolutely continuous and its density is in $H^s(\torus \times \R^d)$.
\end{theorem}	

Notice that Theorem above is stronger than \cite[Theorem 2.9]{bocker.bortolotti} because for $s=0$ the condition is just $|\det E| |\det C|>1$.

Stronger properties related to the action of $\cL$ in a Banach space $\mathcal{B}\subset H^s$, such as spectral gap and exponential decay of correlations, are obtained when the transversality condition is valid and $s>u/2$. 

\begin{theorem}\label{teo.2} Suppose that $C \in \mathcal{C}_s(d;E)$ for $u/2 < s< r-(\frac{u+d}{2}+1)$ and 
$$
\zeta \in (\max\{\|E^{-1}\|^{\frac{1}{1+\log{(1+\frac{d}{2}+u)}}},(|\det E||\det C|\|C^{-1}\|^{2s})^{\frac{1}{2}}\},1).
$$ 

For any $f \in C^r(\torus,\R^d)$ such that $T$ satisfies the transversality condition, there exists an open set $\mathcal{U} \subset C^r(\torus, \R^d)$ containing $f$ such that for every $\tilde{f}\in\mathcal{U}$ there exists a Banach space $\mathcal{B}$ contained in $H^s(\torus\times \R^d)$ and containing $C^{r-1}(D)$ such that the action of the operator $\cL_{\tilde f}$ in $\mathcal{B}$ has spectral gap with essential spectral radius at most $\zeta$. In particular, $T_{\tilde f}$ has exponential decay of correlations in 
some linear space $\mathcal{\tilde B}$ with exponential rate $\zeta$, where  $\mathcal{\tilde B}$ is contained in $\mathcal{B}$ and contains $C^{r-1}(D)$ .
\end{theorem}

\section{Description of the norms $\|\cdot\|^\dagger_\rho$ and $\|\cdot\|_{H^s}$}

In this Section, we define the two main norms that will be used in this work. The Main Inequalities of this paper (Propositions \ref{l.norma cruz}, \ref{l.norma.sobolev} and \ref{togheter}) are stated in terms of these norms.

\subsection{The norm $\|\cdot\|^\dagger_\rho$}\label{ss. set S}

Here we define a norm $\|\cdot\|^\dagger_\rho$ similar to the norms in   \cite{avila.gouezel.tsujii, gouezel.liverani}.

Let $\bc \in I^1$, we define $\mathcal{S}(\bc)$ as the set of $C^r$ transformations $\psi:U_\psi \to \T^u$ such that $U_\psi=\overline{V_\psi}$ for a bounded open set $V_\psi\subset \R^d$, $\psi(U_\psi)  \subset \mathcal{R}_*(\bc)$ 
 and $\|D^\nu\psi(x)\|\le k_\nu$ for $1\le \nu\le r$, for constants $k_1,\cdots,k_r$ that will be chosen appropriately. We define  $\mathcal{S} = \bigcup_{\bc\in I^{1}}\mathcal{S}(\bc)$

Given $\psi\in \mathcal{S}(\bc)$, we denote by $G_\psi = \{(\psi(x),x)\, |\, x \in U_\psi\}$ the graph of $\psi$. 
 For each $\ba \in I^{n}(\bc)$, we denote $(\tilde{G}_\psi)_\ba$ the unique connected component of $T^{-n}(G_\psi)$ which is contained in $\mathcal{R_*}(\ba)\times \R^d$. Moreover, the constants $k_1,\cdots,k_r$ will be chosen such that each set $(\tilde{G}_\psi)_\ba$ is the graph of a transformation $\psi_\ba:U_{\psi_\ba} \to \T^u$ such that $\psi_\ba \in \mathcal{S}$.

Note that $T^n$ is locally written in the form
\begin{equation}
T^n(x,y)=(E^n x, C^n y + S_{\bc,\ba}^n(E^n x) ),
\end{equation}
where $S_{\bc,\ba}^n(z)=\sum_{j=0}^{n-1}C^j f(E_{\bc,\ba}^{-j-1}z)$ is a $C^r$ function with $\|D^j S_{\bc,\ba}^n\|\leq \alpha_0$, $1 \leq j \leq r$.

Given $\psi \in \mathcal{S}(\bc)$ and $(\tilde{G}_\psi)_\ba$, $\ba \in I^n(\bc)$, 
 the inverse branch $T_{\bc,\ba}^{-n}$ is written as
\begin{equation}
T_{\bc,\ba}^{-n} (x,y) = \left((E_{\bc,\ba})^{-n}(x), C^{-n}(y-S^n_{\bc,\ba}(x)) \right).
\end{equation}


Consider the $C^r$ diffeomorphism $g_\ba:U_\psi \to U_{\psi_\ba}$ such that
$T^n(\psi_\ba \circ g_\ba(y),g_\ba(y))=(\psi(y),y)$
for all $y\in U_\psi$. We have $\psi_\ba(y)=E_{\bc,\ba}^{-n} \psi(g_\ba^{-1}(y) )$, $(\tilde{G}_\psi)_\ba=G_{\psi_\ba}$ and $\psi_\ba \in \mathcal{S}$, where the $g_\ba$'s are given by

\begin{equation}\label{g}
g_\ba(y)=C^{-n}(y - S^n_{\bc,\ba}(\psi(y)).
\end{equation}

A useful estimate for the map $g_\ba$ is given in the following.

\begin{claim}\label{claim Qa} The map $g_\ba$ is a $C^r$ diffeomorphism and there exists a $C^{r-1}$ map $Q_\ba: U_{\psi_\ba} \to L(\R^d,\R^d)$ such that $Dg_\ba^{-1}(z)=Q_\ba(z)C^{n}$. Moreover, $\|Q_\ba\|_{C^{r-1}}\le K$ for some constant $K$ depending only $\alpha_0$, $k_1, \dots, k_r$. In particular,
	$$
	\|D^j g_\ba^{-1}(z)\|\le K \|C\|^{n} \quad \text{and} \quad |D^{j-1}\det{Dg_\ba^{-1}(z)}|\le K |\det C|^{n}$$
	for every $z\in U_{\psi_\ba}$ and $1\le j \le r$.
\end{claim}
\begin{proof} The map   $g_\ba$ is one-to-one because
	 $g_\ba(y)=g_\ba(z)$ implies $y-z=S^n_{\bc,\ba}(\psi(y))- S^n_{\bc,\ba}(\psi(z))$. But the estimates  $\|DS^n_{\bc,\ba}\|\le \alpha_0$
	and $\|D\psi\|\le c_1<\alpha_0^{-1}$ implies that $y=z$.
	The expression $Dg_\ba(y)=\sC^{-n}\big(I-DS_{\bc,\ba}^n(\psi(y))D\psi(y)\big)$ implies that $Dg_\ba(y)$ is invertible for every $y\in U_\psi$ due to  $\|DS^n_{\bc,\ba}(\psi(y))D\psi(y)\|\le\alpha_0c_1<1$. This proves that $g_\ba$ is a $C^r$ diffeomorphism.
	
	For every $z\in U_{\psi_\ba}$ 
		we have:
	\begin{align*}
	Dg_{\ba}^{-1} (z) &= (I-DS_{\bc,\ba}^n(\psi(z))D\psi(z))^{-1} C^n    \\
	&=\sum_{k=0}^\infty (DS^n_{\bc,\ba}(\psi( z ))D\psi(  z  ))^k C^n.
	\end{align*}
	
	The result follows
	taking $Q_\ba(z)=\sum_{k=0}^\infty (DS^n_{\bc,\ba}(\psi((g_{\ba})^{-1}(z)))D\psi(g_\ba^{-1}(z)))^k$. 

%
%
\end{proof}


Let us fix the cone field
 \begin{equation}
 \mathcal{C}=\{(u,v) \in T_{(x,y)}\torus\times\R^d\,|\, \|u\|\le \alpha_0^{-1}\|v\|\},
 \end{equation}
which is invariant under $(DT^{-1})_{(x,y)}$ for every $(x,y) \in \torus\times \R^d$.

We suppose that $k_1 \leq \alpha_0^{-1}/2$ and, if necessary, we increase the constants $k_2, \cdots, k_r >0$ in order that the following is valid:
if $\sigma$ is a $u$-dimensional ball contained in a $u$-dimensional plane of $\T^u\times \R^d$ and $\Gamma$ is a connected component of $T^{-q}(\sigma)$ such that its tangent vectors are all in $\mathcal{C}$, then $\Gamma$ is the graphic of an element of $\mathcal{S}$.

For $h\in C^r(D)$ and multi-indexes $\alpha=(\alpha_1,\cdots,\alpha_u)$ and $\beta=(\beta_1,\cdots,\beta_d)$, $|\alpha|+ |\beta| \leq r$,  we denote
\begin{equation}
\partial^\alpha_x\partial^\beta_y h = \frac{ \partial^{|\alpha|+|\beta|} h }{\partial_{x_1}^{\alpha_1}\cdots\partial_{x_u}^{\alpha_u} \partial_{y_1}^{\beta_1} \cdots \partial_{y_d}^{\beta_d}}.
\end{equation}

\begin{definition}\label{norm.dagger}
For $h\in C^r(D)$ and an integer $0\le \rho \le r-1$, we define
\begin{equation}\label{norm}
\|h\|_\rho^{\dagger}=\max_{|\alpha|+|\beta|\le \rho}\sup_{\psi\in \cS}\sup_{\phi\in \cC^{|\alpha|+|\beta|}(U_\psi)}\int\phi(y).\partial_x^{\alpha}\partial_y^{\beta}h(\psi(y),y)\,dy
\end{equation}
where the first supremum is taken over functions $\phi$ with
$\supp(\phi)\subset \Int(U_\psi)$ and $\|\phi\|_{C^{|\alpha|+|\beta|}}\le 1.$
\end{definition}

Clearly,  $\|h\|^{\dagger}_\rho$ is a norm that satisfies:
\begin{equation}\label{eq. normal1normacruz0}
\|h\|_{L^1} \leq \|h\|^\dagger_0
\quad \text{ and } \quad
\|h\|^\dagger_{\rho-1} \leq \|h\|^\dagger_\rho.
\end{equation}


The first main Lasota-Yorke inequality is similar to the ones in \cite{avila.gouezel.tsujii, gouezel.liverani}:

\begin{proposition}[First Main Lasota-Yorke (for $\|\cdot\|^\dagger$)]\label{l.norma cruz}
For any $\delta\in (\|E^{-1}\|, 1)$, there exist constants $K$ and $K(n)$ such that
\begin{equation}\label{primeira.estimativa}
\|\cL^nh\|_{\rho}^{\dagger}\le K\delta^{\rho n}\|h\|_{\rho}^{\dagger}+K(n)\|h\|_{\rho-1}^{\dagger} \quad \text{for $1\le \rho \le r-1$},
\end{equation}
and
\begin{equation}\label{segunda.estimativa}
\|\cL^nh\|_0^{\dagger}\le K\|h\|_0^{\dagger}
\end{equation}
for $n\ge 0$ and $h\in C^r(D)$, where $K(n)$ depends on $n$ but not on $h$.
\end{proposition}

Proposition \ref{l.norma cruz} is proved in Section 4.1.

\subsection{The Sobolev norm $\|\cdot\|_{H^s}$}

Let us remind some facts about the Fourier transform and the  Sobolev norm that shall be used further.

Given $\phi \in C^r(D)$, we define $\hat{\phi}: \Z^u\times \R^d \to \mathbb{C}$ by
\begin{equation}
\hat{\phi}(\xi,\eta) = \int_{\torus\times\R^d} \phi(x,y) e^{- 2\pi i (\langle \xi, x \rangle + \langle \eta, y \rangle )} dxdy
\end{equation}

The Sobolev norm of is defined by $\|\phi\|_{H^s} = \sqrt{\langle \phi, \phi \rangle_{H^s}}$, where
\begin{equation}
\langle \phi_1, \phi_2 \rangle_{H^s} := \sum_{\eta \in \Z^u} \int_{R^d} \hat{\phi}_1(\eta,\xi) \overline{\hat{\phi}_2(\eta,\xi)} (1+|\xi|^2+|\eta|^2)^s d\eta,
\end{equation}
and the Sobolev space $H^s$ is the completion of $C^r(D)$ with respect to this norm. This norm comes from the inner product


An equivalent definition is given by the $L^2$ norm of the derivatives.
For multi-indexes $\alpha=(\alpha_1,\cdots,\alpha_u)$ and $\beta=(\beta_1,\cdots,\beta_d)$, we denote
 $\sigma=(\alpha,\beta)$ and $\partial^\sigma_z h = \partial^\alpha_x\partial^\beta_y h$.
If $s$ is a non-negative integer with $r\geq s$ and $\phi_1,\phi_2 \in C^r(D)$, we define the inner product
\begin{equation}\label{def.sobolev1}
\langle\phi_1, \phi_2 \rangle_{\tilde H^s} = \sum_{|\sigma|\leq s} \langle \partial_z^\gamma \phi_1 , \partial^\gamma_z \phi_2 \rangle_{L^2}.
\end{equation}

If $s$ is not integer, we define $\delta = s - \lfloor s \rfloor \in(0,1)$ and
\begin{equation}\label{def.sobolev2}
\langle\phi_1, \phi_2 \rangle_{\tilde H^s} = \sum_{|\sigma|\leq \lfloor s \rfloor} \langle \partial_z^\gamma \phi_1 , \partial^\gamma_z \phi_2 \rangle_{L^2} +\sum_{|\sigma|=\lfloor s \rfloor} \underset{\R^u\times \R^d}\int \,\,\underset{\R^u\times \R^d}{\int}
\Phi^\sigma(x,y,v,w)dv dw \,\,
dxdy,
\end{equation}
where
\begin{equation*}
\Phi^\sigma(x,y,v,w)
=\frac{ (\partial^{\sigma}\phi_1(x+v,y+w)-\partial^{\sigma}\phi_1(x,y))(\overline{\partial^{\sigma}\phi_2(x+v,y+w)-\partial^{\sigma}\phi_2(x,y)})}{ (|v|^2+|w|^2)^{\frac{u+d}{2} +\delta}} 
\end{equation*} 
is defined considering the extension of $\phi_j$ to $\R^u \times \R^d$ as zero if $(x,y) \notin [0,1]^u \times \R^d \sim \torus \times \R^d$.
This inner product induces the norm $
\|\phi\|^2_{\tilde H^s} = \langle \phi, \phi\rangle_{\tilde H^s}.$
It is a standard fact that these norms are equivalent (see \cite[page 241]{hormander}), that is,
	there exists a constant $K>0$ such that
	\begin{equation}
	\frac{1}{K}\|\phi \|^2_{\tilde H^s} \leq
   \|\phi\|^2_{H^s} 
   \leq K \|\phi\|^2_{\tilde H^s}.
	\end{equation}

\begin{remark}
Through this paper we will introduce several constants $K>0$ depending only on the objects that were fixed before, for simplicity we will keep denoting them as $K$. In the cases that the constant depends on other objects that are not fixed, we will emphasize this dependence. 

\end{remark}

\begin{claim}\label{ineq.sobolev.l1}
	For $0 \leq t < s \leq r$ and $\epsilon>0$, there is a constant $K(\epsilon, t, s)$ such that
	\begin{equation}
	\|\phi\|^2_{H^t} \leq \epsilon \|\phi\|^2_{H^s} + K(\epsilon, t, s) \|\phi\|^2_{L^1}
	\end{equation}
	for every $\phi \in C^r(D)$.
\end{claim}
\begin{proof}
Choose $1<p<+\infty$ such that $(t-\frac{s}{p})(\frac{p}{p-1})\le -(u+d)$ and use the Young's inequality to obtain (putting $t=\frac{s}{p}+t-\frac{s}{p}$ and recall $1/p+1/q=1$ with $q=p/(p-1)$)
\begin{align*}
(1+|\xi|^2+|\eta|^2)^t&=  (1+|\xi|^2+|\eta|^2)^{\frac{s}{p}}(1+|\xi|^2+|\eta|^2)^{t-\frac{s}{p}}\\
                      &\le \epsilon (1+|\xi|^2+|\eta|^2)^s + \tilde K(\epsilon,t,s)(1+|\xi|^2+|\eta|^2)^{(t-\frac{s}{p})(\frac{p}{p-1})}\\
                      &\le \epsilon (1+|\xi|^2+|\eta|^2)^s + \tilde K(\epsilon,t,s)(1+|\xi|^2+|\eta|^2)^{-u-d}.
\end{align*}

So we have
	\begin{align*}
	&\|\phi\|^2_{H^t}=\sum_{\xi\in\Z^u}\int_{\R^d}|\hat{\phi}(\xi,\eta)|^2(1+|\xi|^2+|\eta|^2)^t \,d\eta\\
                    &\le\sum_{\xi\in\Z^u}\int_{\R^d}\epsilon|\hat{\phi}(\xi,\eta)|^2(1+|\xi|^2+|\eta|^2)^s +|\hat{\phi}(\xi,\eta)|^2\tilde K(\epsilon,t,s)(1+|\xi|^2+|\eta|^2)^{-u-d}\,d\eta\\
                    &\leq \epsilon \|\phi\|^2_{H^s} + K(\epsilon,t,s) \|\hat{\phi}\|^2_{L^\infty}
                    \leq \epsilon \|\phi\|^2_{H^s} + K(\epsilon,t,s) \|\phi\|^2_{L^1}.
	\end{align*}

\end{proof}

\begin{remark}\label{composition}
	Given a multi-index $\sigma$, 
	for every $f:D\to \R$ and $g:\torus\times \R^d \to \torus \times \R^d$ infinitely many times  differentiable, we have
	\begin{equation}\label{eq.cadeia}
	\partial^\sigma(f\circ g)(x) = \sum_{1\le|\sigma'|\leq |\sigma|} \partial^{\sigma'}f(g(x)) \cdot Q_{\sigma,\sigma'}(g;x),
	\end{equation}
	where $Q_{\sigma,\sigma'}(g;\cdot)$ is a homogeneous polynomial of degree $|\sigma'|$ in the  derivatives of $g_1, \dots, g_{u+d}$ until order $|\sigma|-|\sigma'|+1$.

As a consequence, given $F:U\subset\torus\times \R^d \to F(U) \subset\torus\times \R^d$ of class $C^r$ and $u: U\subset\torus\times \R^d\to \R$ a function in $H^s$ for some $s\le r$  supported in $F(U)$,
	 there exists a constant $K=K(F)$ depending on $F$ and its derivatives up to order $\lfloor s \rfloor$ such that
	 \begin{equation}\label{eq.normacomp}
	 \|u\circ F \|_{H^s} \leq K(F) \|u\|_{H^s}
	 \end{equation}
\begin{proof}
	The formula for the  derivative of the  composition in \eqref{eq.cadeia} can be seen in \cite{fraenkel} 
and the estimate in \eqref{eq.normacomp} is an immediate consequence     using the expressions for $\|\cdot\|_{\tilde H^s}$.
\end{proof}
\end{remark}

When $\phi_1$ and $\phi_2$ have disjoint supports, then we have an  estimative for $\langle \phi_1, \phi_2 \rangle$.

\begin{claim}\label{support}
	For $\epsilon>0$, there exists a constant $K(\epsilon, s)$ such that
	\begin{equation}
	|\langle \phi_1,\phi_2 \rangle_{\tilde H^s} | \leq K(\epsilon,s) \|\phi_1\|_{L^1} \|\phi_2\|_{L^1}
	\end{equation}
	for every $\phi_1, \phi_2 \in C^r(D)$ whose support are disjoints and the distance between them is greater than $\epsilon$.
\end{claim}
\begin{proof}
	If $s$ is integer, by \eqref{def.sobolev1} the inner product is $0$.
	
	If $s$ is not integer then we use \eqref{def.sobolev2}, the disjointness of the  supports and change of variables to obtain  
	\begin{align*}
	\langle \phi_1,\phi_2\rangle_{\tilde H^s} =
	-2 \sum_{|\sigma|\le \lfloor s \rfloor}\int_{\R^u \times \R^d}\int_{\R^u \times \R^d} \Theta^\sigma(x,y,v,w)dv dw \,\,
	dxdy,
	\end{align*}
where
$$
\Theta^\sigma(x,y,v,w)
=\frac{ \partial^{\sigma}\phi_1(x+v,y+w)\overline{\partial^{\sigma}\phi_2(x,y)}}{ (|v|^2+|w|^2)^{\frac{u+d}{2} +\delta}}. 
$$

Integrating by parts $\lfloor s \rfloor$ times in $(v,w)$ according to each index in $\sigma$, changing variables and integrating by parts again $\lfloor s \rfloor$ times, we obtain:
	\begin{equation*}
	\langle \partial^\sigma\phi_1,\partial^\sigma\phi_2\rangle_{\tilde H^s} =\int_{\R^u \times \R^d} \int_{\R^u \times \R^d} 
	\frac{ \phi_1(x+v,y+w)\overline{\phi_2(x,y)}B(v,w)}{ (|v|^2+|w|^2)^{\frac{u+d}{2} +\delta+2\lfloor s \rfloor}} dvdw \,\,dxdy
,
	\end{equation*}
where $B(v,w)$ is a polynomial of order $2\lfloor s \rfloor$.
	The proof follows noticing that the integrand vanish if $|v|^2+|w|^2\le \epsilon^2.$

\end{proof}

\begin{claim}\label{interpolacao}
Given   $0\leq s_0 \leq s_1$, a linear operator $L:H^{s_0}  \to H^{s_0}$ such that  $L(H^{s_1}) \subset H^{s_1}$     and constants  $A_0, A_1$  such that:
$$\| L(u) \|_{H^{s_0}} \leq A_0 \|u \|_{H^{s_0}} \quad \text{ and } \quad \| L(u) \|_{H^{s_1}} \leq A_1 \|u \|_{H^{s_1}}.$$

Then $L(H^{s_\theta} )  \subset H^{s_\theta}$ for $s_\theta=(1-\theta)s_0+\theta s_1$, $\theta \in [0,1]$, and
\begin{equation}\label{interpolation}
\| L(u) \|_{H^{s_\theta}} \leq A_0^{1-\theta}A_1^\theta \|u \|_{H^{s_\theta}}.
\end{equation}
\end{claim}
\begin{proof}
This corresponds to Theorem 22.3 in \cite{tartar}.
\end{proof}

The second main Lasota-Yorke Inequality of this work corresponds to the following.

\begin{proposition}[Second Main Lasota-Yorke (for Sobolev norm)]\label{l.norma.sobolev} There exist a constant $B_1$, independent of $q$, and $K(q)$ such that for every $\phi\in C^r(D)$ and every integer $\rho_0$ with $s+\frac{u+d}{2} < \rho_0 \leq r-1$, we have
\begin{equation}
\|\mathcal{L}^q \phi\|_{H^s}^2 \leq B_1 \frac{\tau(q)}{|\det E|^q |\det C|^{q}\fm(C)^{2sq}} \|\phi\|^2_{H^s} + K(q) \| \phi \|_{H^s} \|\phi\|^\dagger_{\rho_0}.
\end{equation}
\end{proposition}

Proposition \ref{l.norma.sobolev} is proved in Section 4.2.

\section{Proof of the Main Inequalities}

\subsection{First Lasota-Yorke (for $\|\cdot\|^\dagger$)}

\begin{proof}[Proof of Proposition \ref{l.norma cruz}] We will prove supposing that $C$ is in the Jordan canonical form. In particular $\R^{d}=E_1\oplus E_2 \oplus \dots \oplus E_k$, where the $E_j$'s are subspaces generated by vectors of the canonical basis, each $E_j$ is invariant by $C$ and $C_{|E_j}$ has all eigenvalues with the same absolute value $\lambda_j>0$.

\begin{claim}\label{jordan.form}
	It is enough to prove Lemma \ref{l.norma cruz} supposing that $C$ is in the Jordan canonical form.
\end{claim}
\begin{proof}
	Consider $P: \R^d \to \R^d$ be an invertible linear operator and consider the transformation $\tilde{T}: \T^u \times \R^d \to \T^u \times \R^d$ given
	by $\tilde{T}(x,y)=(Ex, PCP^{-1}y+Pf(x))$ and the associated Perron-Frobenius operator
	$$
	\mathcal{\tilde L}\tilde h(x,y)=(|\det E||\det C|)^{-1}\sum_{\tilde{T}(x',y')=(x,y)}\tilde{h}(x',y').
	$$
	
	Notice that the transformations $T$ and $\tilde T$ are linear conjugated by $\mathcal{P}: \T^u \times \R^d \to \T^u \times \R^d$, $\mathcal{P}(x,y)=(x,Py)$, that is $\mathcal{P}\circ T=\tilde{T}\circ \mathcal{P}$. Moreover, defining $\tilde{D}=\mathcal{P}(D)$ and, for $\tilde{h}\in C^{r}_0(\tilde{D})$, the norm
	\begin{equation}
	\|\tilde h\|_\rho^{\dagger\dagger}=\max_{0\le |\alpha|+|\beta|\le \rho}\sup_{\psi\in \mathcal{S}}\sup_{\phi\in C^{|\alpha|+|\beta|}(\psi)}\int (\phi\circ P^{-1})(y)\partial^{\alpha}_x\partial^{\beta}_y\tilde{h}(\psi\circ P^{-1}(y),y)\,dy,
	\end{equation}
	then the operator $\mathcal{U}:(C^{r}_0(D),\|.\|^{\dagger}_\rho) \to (C^{r}_0(\tilde D),\|.\|^{\dagger\dagger}_\rho)$ given by
	$\mathcal{U}(h)=h\circ \mathcal{P}^{-1}$ is a bounded isomorphism, that is, there is a constant $B>0$ such that $\|\mathcal{U}(h)\|_{\rho}^{\dagger\dagger}\le B\|h\|_{\rho}^{\dagger}$ and $\|\mathcal{U}^{-1}(\tilde h)\|_{\rho}^{\dagger}\le B\|\tilde h\|_{\rho}^{\dagger\dagger}$.	
	Clearly, it is valid that
$	\mathcal{U}\circ \mathcal{L}=\mathcal{\tilde L}\circ \mathcal{U}$.
	
	So if for some constants $a\ge 0$ and $b\ge 0$ we have
	\begin{equation}
	\|\tilde{\cL}^n\tilde h\|_{\rho}^{\dagger\dagger}\le a\|\tilde h\|_{\rho}^{\dagger\dagger}+b\|\tilde h\|_{\rho-1}^{\dagger\dagger},
	\end{equation}
	then
	\begin{equation}
	\|\cL^n h\|_{\rho}^{\dagger}\le aB^2\| h\|_{\rho}^{\dagger}+bB^2\|h\|_{\rho-1}^{\dagger}.
	\end{equation}

\end{proof}
	In the rest of this proof, we suppose  that $C:\R^d \to \R^d$ is in the Jordan form. Notice that $E_i \perp E_j$ for $i\neq j$ and therefore all $E_j$ are invariants by $C^*$ and $C^{*}_{|E_j}$ has all eigenvalues with the same absolute value $\lambda_j>0$.
	In these conditions for any canonical vector $\epsilon_l$ we have
\begin{equation}
	\lim_{n \to \pm\infty}\frac{1}{n}\log\|C^{n}\epsilon_l\|=\lim_{n \to \pm\infty}\frac{1}{n}\log\|(C^*)^{n}\epsilon_l\|=\log\lambda_l
\end{equation}
	and, in particular,
\begin{equation}\label{log}
	\lim_{n \to \pm\infty}\frac{1}{n}\log(\|C^{-n}\epsilon_l\|\|(C^*)^{n}\epsilon_l\|)=0.
\end{equation}

Denote $\{e_1, \dots, e_u\}$ the canonical basis of $\R^u$ and $\{\epsilon_1, \dots, \epsilon_d\}$ the canonical basis of $\R^d$. 
 We have the following formula for the derivatives of  $\mathcal{L}^nh(x,y)$.

\begin{claim}\label{derivative} If $1\le |\alpha|+|\beta|=\rho\le r-1$, then
\begin{equation}
\partial_x^{\alpha} \partial_{y}^{\beta}(\mathcal{L}^n h)(x,y) =  \sum_{(x',y') \in T^{-n}(x,y)} \sum_{|a|+|b| \leq \rho} \frac{ \partial_x^{a} \partial_{y_1}^{b_1}\cdots \partial_{y_d}^{b_d} h (x',y') \cdot Q_{\alpha,\beta,a,b,n}(x)}{(|\det E| |\det C|)^{n}  \cdot \|E^{-n}\|^{-|a|} \prod_{l=1}^{d}\|(C^{-n})^* \epsilon_{l}\|^{-b_l} },
\end{equation}
where $\beta=(b_1,\dots, b_l)$ and the functions $Q_{\alpha,\beta,a,b,n}$ are of class $C^{r-1-|\alpha|-|\beta|+|a|+|b|}$ 
 and there exists a constant $K$ such that $\|Q_{\alpha,\beta,a,b,n}\|_{C^{|a|+|b|}} \leq K$ for every $n \geq 0$, $\alpha$, $\beta$, $a$ and $b$ with $|a|+|b| \le |\alpha|+|\beta| \leq \rho \leq r-1$.
\end{claim}

\begin{proof}
By induction in $\rho$, noticing that the inverse branch $T_{\bc,\ba}^{-n}$ is locally written as 
\begin{equation}\label{inverse.branch}
T_{\bc,\ba}^{-n}(x,y)=(E_{\bc, \ba}^{-n}x,C^{-n}(y-S_{\bc,\ba}(x)) ),
\end{equation}
where $S_{\bc,\ba}^n(x)=\sum_{j=0}^{n-1} C^j f(E_{\bc,\ba}^{-j-1}y)$ is a $C^r$ function with $\|D^j S_{\bc,\ba}^n\|\leq \alpha_0$, $1 \leq j \leq r$.
\end{proof}

Using this formular, for $\psi \in \Omega$, $\phi \in C^r(\psi)$ with $\|\phi\|_{C^\rho}\leq 1$, and considering  $\psi_1,\cdots,\psi_N$, $g_1,\cdots,g_N$ such that $T^n(\psi_i(g_i(y),y)=(\psi(y),y)$, we have:

 \begin{align*}
 &\int \phi(y)\cdot  \partial_x^{\alpha} \partial_{y}^{\beta} (\mathcal{L}^n h)(\psi(y),y) dy
 \\&=  \sum_{1\leq i \leq N} \int \sum_{|a|+|b|\leq \rho} \frac{\phi(y)  \partial_x^{a} \partial_{y_1}^{b_1}\cdots \partial_{y_d}^{b_d}   h(\psi_i(g_i(y)),g_i(y) )\cdot Q_{\alpha,\beta,a,b,n}(\psi(y))}{(|\det E| \cdot |\det C|)^{n} \|E^{-n}\|^{-|a|}  \prod_{l=1}^{d}\|(C^{-n})^* \epsilon_l\|^{-b_l} } dy \\
 &= \sum_{1\leq i \leq N}\sum_{|a|+|b|\leq \rho}  \int \frac{\Psi_{\alpha,\beta,a,b,n;i}(y') \cdot   \partial_x^{a} \partial_{y_1}^{b_1}\cdots \partial_{y_d}^{b_d}  h(\psi_i(y'),y')}{(|\det E| |\det C|)^n \|E^{-n}\|^{-|a|} \prod_{l=1}^{d}\|(C^{-n})^* \epsilon_{l}\|^{-\beta_l} } dy'
 \end{align*}
 where $\Psi_{\alpha,\beta,a,b,n;i}(y')=\phi(g_i^{-1}(y')) \cdot Q_{\alpha,\beta,a,b,n}((\psi_i\circ g_i^{-1})(y')) \cdot |\det Dg_{i}^{-1}(y')|.$

 Note that $\Psi_{\alpha,\beta,a,b,n;i}$ has $C^{|a|+|b|}$-norm uniformly bounded by some constant $K$, depending on the constants $k_1,k_2,\cdots,k_r$ on the definition of $\Omega$ but not on $h$. In particular, we have
 \begin{equation}\label{eq.estimate1}
\sum_{1\leq i \leq N}\sum_{|a|+|b|< \rho}  \int \frac{\Psi_{\alpha,\beta,a,b,n;i}(y') \cdot   \partial_x^{a} \partial_{y_1}^{b_1}\cdots \partial_{y_d}^{b_d}  h(\psi_i(y'),y')}{(|\det E| |\det C|)^n \|E^{-n}\|^{-|a|} \prod_{l=1}^{d}\|(C^{-n})^* \epsilon_{l}\|^{-\beta_l} } dy' \leq K(n) \|h\|^\dagger_{\rho-1}.
 \end{equation}

We will estimate the sum
\begin{equation}\label{eq.estimate2}
\sum_{1\leq i \leq N}\sum_{|a|+|b|= \rho}  \int \frac{\Psi_{\alpha,\beta,a,b,n;i}(y') \cdot   \partial_x^{a} \partial_{y_1}^{b_1}\cdots \partial_{y_d}^{b_d}  h(\psi_i(y'),y')}{(|\det E| |\det C|)^n \|E^{-n}\|^{-|a|} \prod_{l=1}^{d}\|(C^{-n})^* \epsilon_{l}\|^{-\beta_l} } dy'.
\end{equation}

To integrate by parts, fix $i_0\in \{1,\dots,N\}$ and multi-index $(a,b)$ such that $|a|+|b|=\rho$ and note that for $b_1\ge 1$
 \begin{align*}
 \partial_x^a \partial_{y_1}^{b_1}\cdots  &\partial_{y_d}^{b_d}  h(\psi_{i_0}(y'),y')=\partial_{y_{1}}[\partial_x^a \partial_{y_1}^{b_1-1} \partial_{y_2}^{b_2} \cdots \partial_{y_d}^{b_d}  h(\psi_{i_0}(y'),y')]
 \\
 &-\sum_{j=1}^u \partial_{x_j} \big( \partial_x^a \partial_{y_1}^{b_1-1}\partial_{y_2}^{b_2} \cdots \partial_{y_d}^{b_d} h(\psi_{i_0}(y'),y') \big) \langle D\psi_{i_0}(y')\epsilon_k,e_j\rangle,
 \end{align*}

 If $b_1=1$ then the partial derivative with respect to $y_1$ disappear, otherwise we repeat the process until the partial derivative with respect to $y_1$ disappear. So:
 \begin{align*}
 \partial_x^a \partial_{y_1}^{b_1}\cdots  &\partial_{y_d}^{b_d}  h(\psi_{i_0}(y'),y')=\partial_{y_{1}}[\partial_x^a \partial_{y_1}^{b_1-1} \partial_{y_2}^{b_2} \cdots \partial_{y_d}^{b_d}  h(\psi_{i_0}(y'),y')]+\\
 &\sum_{m=1}^{b_1-1}(-1)^{m}\sum_{|a'|=m} \partial_{y_1}\big( \partial_x^{a+a'} \partial_{y_1}^{b_1-m-1}\partial_{y_2}^{b_2} \cdots \partial_{y_d}^{b_d} h(\psi_{i_0}(y'),y') \big) \prod_{j=1}^u(\langle D\psi_{i_0}(y')\epsilon_1,e_j\rangle)^{a'_j}\\
 &+ (-1)^{b_1}\sum_{|a'|=b_1} \big( \partial_x^{a+a'} \partial_{y_2}^{b_2}\cdots  \partial_{y_d}^{b_d} h(\psi_{i_0}(y'),y') \big) \prod_{j=1}^u(\langle D\psi_{i_0}(y')\epsilon_1,e_j\rangle)^{a'_j},
 \end{align*}
which may rewritten as
\begin{align*}
 \partial_x^a \partial_{y_1}^{b_1}\cdots  &\partial_{y_d}^{b_d}  h(\psi_{i_0}(y'),y')=\\
 &\sum_{|a'|=0}^{b_1-1}(-1)^{|a'|} \partial_{y_1}\big( \partial_x^{a+a'} \partial_{y_1}^{b_1-|a'|-1}\partial_{y_2}^{b_2} \cdots \partial_{y_d}^{b_d} h(\psi_{i_0}(y'),y') \big) \prod_{j=1}^u(\langle D\psi_{i_0}(y')\epsilon_1,e_j\rangle)^{a'_j}\\
 &+\sum_{|a'|=b_1}(-1)^{|a'|} \big( \partial_x^{a+a'} \partial_{y_2}^{b_2}\cdots  \partial_{y_d}^{b_d} h(\psi_{i_0}(y'),y') \big) \prod_{j=1}^u(\langle D\psi_{i_0}(y')\epsilon_1,e_j\rangle)^{a'_j},
 \end{align*}

Applying repeatedly the same process to the last sum, but considering  derivatives with respect to $y_2, \dots, y_d$ successively, we obtain that:

$$\partial_x^a \partial_{y_1}^{b_1}\cdots  \partial_{y_d}^{b_d}  h(\psi_{i_0}(y'),y')=\sum_{k=1}^{d}F_k(y')+F(y'),$$
where
 \begin{align*}
& F_k(y')=\sum_{|a^{(1)}|=b_1} \dots \sum_{|a^{(k-1)}|=b_{k-1}}\sum_{|a^{(k)}|=0}^{b_k-1} (-1)^{\sum_{j=1}^k|a^{(j)}|}(\partial_{y_k}H_{a^{(1)}, \dots, a^{(k)}}(y'))G_{a^{(1)}, \dots, a^{(k)}}(y'),\\
&H_{a^{(1)}, \dots, a^{(k)}}(y')=\partial_x^{a+a^{(1)}+\dots +a^{(k)}} \partial_{y_k}^{b_k-|a^{(k)}|-1}\partial_{y_{k+1}}^{b_{k+1}} \cdots \partial_{y_d}^{b_d} h(\psi_{i_0}(y'),y').\\
&G_{a^{(1)}, \dots, a^{(k)}}(y')=\prod_{l=1}^{k}\prod_{j=1}^u(\langle D\psi_{i_0}(y')\epsilon_l,e_j\rangle)^{a^{(l)}_j}
\end{align*}
and
\begin{align*}
 &F(y')=\sum_{|a^{(1)}|=b_1} \dots \sum_{|a^{(d)}|=b_{d}} (-1)^{|b|}\tilde{H}_{a^{(1)}, \dots, a^{(d)}}(y'))\tilde{G}_{a^{(1)}, \dots, a^{(d)}}(y'),\\
 &\tilde{H}_{a^{(1)}, \dots, a^{(d)}}(y')=\partial_x^{a+a^{(1)}+\dots +a^{(d)}}h(\psi_{i_0}(y'),y'),\\
 &\tilde{G}_{a^{(1)}, \dots, a^{(d)}}(y')=\prod_{l=1}^{d}\prod_{j=1}^u(\langle D\psi_{i_0}(y')\epsilon_l,e_j\rangle)^{a^{(l)}_j}.
\end{align*}

Integrating by parts it is easy to note that
\begin{equation}\label{eq.estimate2.1}
\int \frac{\Psi_{\alpha,\beta,a,b,n;i_0}(y') \cdot   \sum_{k=1}^{d}F_k(y')}{(|\det E| |\det C|)^n \|E^{-n}\|^{-|a|} \prod_{l=1}^{d}\|(C^{-n})^* \epsilon_{l}\|^{-\beta_l} } dy'\le K(n) \|h\|^{\dagger}_{\rho-1}.
\end{equation}

By Claim \ref{claim Qa}, the derivatives 
$ D^j g_\ba^{-1}(z)=D^{j-1}Q_\ba(z)C^{n}$, $D^{j-1} \det{Dg_\ba^{-1}(z)} = D^{j-1}(\det$ $Q_\ba(z)) \det C^{n}$ and 
$  D\psi_\ba(z)=E^{-n}D\psi(g_\ba^{-1}(z))Q_\ba(z)C^n$ are uniformly bounded by some constant $K$, since $Q_\ba$ is $C^{r-1}$ uniformly bounded.

So we have
\begin{equation}
\|\Psi_{\alpha,\beta,a,b,n;i_0}(y')\|_{C^{\rho}}\le K|\det C|^{n}
\end{equation}
 and 
\begin{equation}
\|\tilde{G}_{a^{(1)}, \dots, a^{(d)}}(y')\|_{C^{\rho}}\le K \|E^{-n}\|^{|b|}\prod_{=1}^d\|C^{n}\epsilon_l\|^{b_l},
\end{equation}
 hence
\begin{equation}\label{eq.estimate2.2}
\int \frac{\Psi_{\alpha,\beta,a,b,n;i_0}(y') \cdot   F(y')}{(|\det E| |\det C|)^n \|E^{-n}\|^{-|a|} \prod_{l=1}^{d}\|(C^{-n})^* \epsilon_{l}\|^{-\beta_l} } dy'\le K(n)\|h\|^{\dagger}_{\rho},
\end{equation}
where
\begin{equation}
K(n)=\frac{K}{(|\det E|)^n \|E^{-n}\|^{-\rho} \prod_{l=1}^{d}(\|(C^{-n})^* \epsilon_{l}\|\|C^{n}\epsilon_l\|)^{-\beta_l} }.
\end{equation}

Therefore, by \eqref{eq.estimate2.1} and \eqref{eq.estimate2.2}, we conclude that \eqref{eq.estimate2} is bounded by
$$
K(n)\|h\|_{\rho-1}^{\dagger}+ K\|E^{-n}\|^{\rho} \prod_{l=1}^{d}(\|(C^{-n})^* \epsilon_{l}\|\|C^{n}\epsilon_l\|)^{\beta_l}\|h\|_{\rho}^{\dagger}.
$$

From \eqref{log}  we have that  $\ds\frac{\log(\|(C^{-n})^* \epsilon_{l}\|\|C^{n}\epsilon_l\|)}{n}$ converges to zero, which implies \eqref{primeira.estimativa}. The estimate in \eqref{segunda.estimativa} is analogous and easier.
\end{proof}

\subsection{Second Lasota-Yorke (for Sobolev norm)}

Through this Section we fix an integer $q$ and fix $p$ such that $\tau(q,\tilde{p})=\tau(q)$ for every $\tilde{p} \geq p$.

Since $\cL \phi (x) = |\det DT|^{-1} \sum \phi \circ T_{\bc,\ba}^{-1}(x)$, Remark \ref{composition} and \eqref{inverse.branch} imply that $\cL$ is a bounded operator in $H^s$, that is
$$
	\|\cL (\phi)\|_{H^s} \leq  K \|\phi\|_{H^s}.
$$

Let us consider the dual cone fields
\begin{equation}
\cC^*=\{(u,v)\in \R^u\times\R^d\,|\, \|v\|\le \alpha_0^{-1} \|u\|\}
\end{equation}
and
\begin{equation}
\cC^*_1 = \{ (u,v)\in \R^u\times\R^d\,|\, \|v\|\le \frac{9}{10} \alpha_0^{-1} \|u\| \}.
\end{equation}

Notice that for all $(\xi_0,\eta_0)\neq 0$ in $\cC_1^*$ there is a $u$-dimensional subspace $W_0$ contained in $\cC_1^*$ such that $(\xi_0,\eta_0)\in W_0$. Indeed, it is enough to take
$$W_0=[\{(\xi_0,\eta_0), (\xi_1,0), \dots, (\xi_{u-1},0)\}],$$ where $\{\frac{\xi_0}{\|\xi_0\|},\xi_1,\dots,\xi_{u-1}\}$ is an orthonormal base of $\R^u$.

By continuity of $(x,y) \mapsto (DT^q_{(x,y)})^*$
and noticing that this map does not depend on $y$, it follows that if $(DT^q_{(x_0,y_0)})^*(\xi,\eta) \in \cC^*_1$ then there exists a $u$-dimensional subspace $W$ such that $(\xi,\eta)\in W$ and a constant $R=R(q)>0$ such that
$
(DT^q_{(x,y)})^*W \subset \cC^*
$
 for every $x\in B(x_0,R)$ and $y\in \R^d$. More precisely, we conclude that
\begin{equation}
 (DT^q_{(x,y)})^*((DT^q_{(x_0,y_0)})^*)^{-1} \cC_1^* \subset \cC^*.
 \end{equation}

Consider $p$ sufficiently large such that $\cR_*(\bc\ba)\subset B(x,R)$ for all $x\in \cR(\bc\ba)$, where $R=R(q)$ is given as above.

The following lemma gives a comparision between   $\|\phi\|^\dagger_{\rho}$ 
and the  Fourier transform of iterates of $\cL^q h(\xi,\eta)$ when $(DT^q)^*(\xi,\eta)$ is in $\mathcal{C}^*$.
The main point behind this comparison between is that the condition $(DT^q)^*(\xi,\eta) \in \mathcal{C}^*$ {allows to consider $\sigma$ with $(\xi,\eta)\in \sigma^\perp$ such that $\sigma^\perp = T^q(\tilde{\sigma})$ with $\tilde{\sigma}\in\Omega$.}

\begin{lemma}\label{estimar.transf.vs.l1}
Let $\rho_0$ be an integer with $s+1<\rho_0\le r-1$. Let $\ba\in I^{q}$ and $\bc\in I^{p}$, and $\chi:\T^u\times \R^d \to \R$ a $C^\infty$ function supported on $\cR(\bc\ba)\times \R^d$.
If $0\neq(\xi,\eta)\in \Z^u\times \R^d$ satisfies $(DT^q_{x_0})^*(\xi,\eta)\in \cC_1^*$ for some $x_0\in \cR(\bc\ba)\times \R^d$. Then, for any $\phi\in C^r(D)$,
\begin{equation}
  (1+\|\xi\|^2+\|\eta\|^2)^{\frac{\rho_0}{2}}|\cF(\cL^q(\chi . \phi))(\xi,\eta)|\le K(\chi,q)\|\phi\|^{\dagger}_{\rho_0}
\end{equation}
where $K(\chi,q)$ depends only on $\chi$ and $q$.
\end{lemma}

\begin{proof} We will consider a $u$-dimensional subspace $W$ as described above satisfying $(DT_x^q)^*W\subset \cC^*$, for all $x\in B(x_0,R)\supset \cR(\bc\ba)$ and $(\xi,\eta)\in W$.


Let $0\neq(\xi,\eta)\in W\cap \Z^u\times \R^d$, then the standard property of Fourier transform $\Fc(\partial_{x_k}u)=i \xi_k \Fc u$ gives:
\begin{equation}
|\Fc(\cL^q(\chi \phi))(\xi,\eta)| (\|\xi\|^{\rho_0} + \|\eta\|^{\rho_0}) \leq K |\Fc(\partial^{\rho_0}_{x_j}\cL^q(\chi\phi))(\xi,\eta)|,
\end{equation}
where the $\rho$ derivatives are taken with respect to the variable $x_j$ ($\xi_j$ or $\eta_j$) that has greatest absolute value ($|x_j|=\max \{|\xi_j|, |\eta_j|\}$). 

Define the partition $\Gamma$ of $D\cap(\mathcal{R}(\bc)\times\R^d)$  formed by the intersections $\sigma$ of $D\cap(\mathcal{R}(\bc)\times\R^d)$ with the $d$-dimensional affine manifolds orthogonal to $W$.

Since the support of $\cL^q(\chi \phi)$ is contained in $D\cap (\cR(\bc)\times \R^d)$, Rokhlin's disintegration theorem gives:
\begin{align*}
|\Fc(\partial^{\rho_0}_{x_j}\cL^q(\chi\phi))(\xi,\eta)| &\leq \int_{\cR(\bc)\times \R^d} | \partial^{\rho_0}_{x_j}\cL^q(\chi\phi))(x,y) |  dm \\
&\leq \int_\Gamma \int_{\sigma } | \partial^{\rho_0}_{x_j}\cL^q(\chi\phi))(x,y) | dm_\sigma (x,y) d\hat{m}(\sigma)\\
&\leq \hat{m}(\Gamma) \sup_{\sigma \in \Gamma}\int_{\sigma} |\partial^{\rho_0}_{x_j}\cL^q(\chi\phi)(x,y)|dm_\sigma.
\end{align*}

Each $m_\sigma$ above is the $d$-dimensional Lebesgue measure on $\sigma$ and  
$\hat m$ is identified with the $u$-dimensional Lebesgue measure on the set of the points $w\in W$ such that $(w+W^\perp)\cap \sigma \neq \emptyset$ for some $\sigma\in\Gamma$. In particular, $\hat{m}(\Gamma)$ is finite, because the set of points $w\in W$ such $(w+W^\perp)\cap \sigma\neq \emptyset$ for some $\sigma\in\Gamma$ is bounded.


For each $\sigma \in \Gamma$, there is a unique $\widetilde{\sigma}$ contained in $\cR(\bc\ba)\times \R^d$ such that $T^q(\widetilde{\sigma})=\sigma$.
For $x\in \widetilde{\sigma}$ and $(u,v)$ tangent to $\sigma$ at $T^q(x)$, we have
$$
0=\langle (u,v) ,(w_1,w_2) \rangle=\langle (DT^q_x)^{-1}(u,v) ,(DT^q_x)^*(w_1,w_2) \rangle
$$
for all $(w_1,w_2)\in W$.

Since $(DT^q_x)^*W$ is a $u$-dimensional subspace contained in $\cC^*$, we have $(DT^q_x)^{-1}(u,v)\in\cC$. So, we conclude that
 $\tilde{\sigma}=T^{-q} \sigma \cap (\cR(\bc\ba)\times \R^d)$ is the graph of some $\tilde{\psi}$ in $\cS$.

Since $\chi$ is supported in $\mathcal{R}(\bc\ba)\times \R^d$, 
we have that  $\cL^q(\chi\phi) = \frac{(\chi\phi)\circ b}{|\det DT^q|}$ 
for the inverse branch $g:\cR(\bc)\times \R^d \to \cR(\bc\ba)\times \R^d$ of the restriction of $T^{q}$  to $\R(\bc\ba)\times \R^d$. Then
\begin{align*}
|\det DT^{q}||\partial^{\rho_0}_{x_j}\cL^q(\chi\phi)(x,y)|&=|\partial^{\rho_0}_{x_j}((\chi\phi)\circ g)(x,y)| \\ &= |\sum k_{\alpha,\beta,\gamma} \partial^\alpha \chi(g(x,y)) \partial^\beta \phi(g(x,y))\partial^{\gamma} g (x,y)|\\
&\leq K(\chi)K(b) \sum_{\beta}|\partial^{\beta}\phi(g(x,y))|
\end{align*}

 Integrating and changing variables, we obtain:

\begin{align*}
\int_{\sigma} |\partial^{\rho_0} \cL^q(\chi \phi)(x,y)|dm_{\sigma} &\leq |\det DT^q|^{-1}K(\chi,q) \sum_\beta \int_\sigma |\partial^\beta \phi(b(x,y))|dm_{\sigma} \\
&\leq K(\chi,q) \int_ {\tilde{\sigma}}|\partial^\beta \phi(\tilde{x},\tilde{y})|dm_{\tilde{\sigma}} \leq K(\chi,q) \|\phi\|^\dagger_{\rho_0}
\end{align*}

Putting it together, we have that
$$|\Fc(\cL^q(\chi \phi))(\xi,\eta)|( \|\xi\|^{\rho_0} + \|\eta\|^{\rho_0}) \leq K  \hat{m}(\Gamma) K(\chi,q) \|\phi\|^\dagger_{\rho_0}.$$

Finally, the result follows noticing 
 that $|\cF(\cL^q(\chi \phi))(\xi,\eta)|\le \|\phi\|_{L^1}\le K\|\phi\|^{\dagger}_{\rho_0}$ and  $(1+|\xi|^2+|\eta|^2)^{\frac{\rho_0}{2}}\le K (1+|\xi|^{\rho_0}+|\eta|^{\rho_0})$.

\end{proof}

One Lemma concerning the transversality that shall be used in the proof of the Lasota-Yorke inequality is the following:

\begin{lemma}\label{transversality.lemma}
	Let $(\xi,\eta)\in \Z^u\times \R^d\setminus \{0\}$. If $\ba$ is transversal to $\bb$ on $\cR_*(\bc)$ then either $(DT_x^q)^*(\xi,\eta)\in \cC_1^*$ for all
$x\in \sE^{-q}_{\bc,\ba}(\cR_*(\bc))$ or $(DT^q_x)^*(\xi,\eta)\in \cC_1^*$ for all
 $x\in \sE^{-q}_{\bc,\bb}(\cR_*(\bc))$.
\end{lemma}

\begin{proof} 
%
%
%

	Note that if $E^q(x_\ba)=x$ for some $x_\ba \in \sE^{-q}_{\bc,\ba}(\cR_*(\bc))$
then
	$$
	(DT^q(x_\ba,y))^*=\left(
	\begin{array}{cc}
	(E^q)^* & (E^q)^*(DS_\bc(x,\ba))^* \\
	0 & (C^q)^*\\
	\end{array}
	\right).
	$$
	
	Supposing that  $(DT^q(x_\ba,y))^*(\xi,\eta)\not\in\cC_1^*$ for some $x\in \sE^{-q}_{\bc,\ba}(\cR_*(\bc))$
, then we claim that  $(DT^q(x_\bb,y))^*(\xi,\eta)\in\cC_1^*$ for all
	$x_\bb\in\sE^{-q}_{\bc,\bb}(\cR_*(\bc))$. 
	
	In fact, if both vectors are not in $\cC_1^*$, then
	$\|(\sC^q)^*\eta\|> 9/10 \alpha_0^{-1}\|(\sE^q)^*\xi+(\sE^q)^*(DS_\bc(x,\ba))^*\eta\|$ and $\|(\sC^q)^*\eta\|> 9/10 \alpha_0^{-1}\|(E^q)^*\xi+(E^q)^*(DS_\bc(\tilde{x},\bb))^*\eta\|$. Then, summing and using triangular inequality, we have that 
	$$
	2\|(\sC^q)^*\eta\|>9/10 \alpha_0^{-1}\|(\sE^q)^*(DS_\bc(x,\ba)-DS_\bc(\tilde x,\bb))^*\eta\|.
	$$

	On the other hand, the transversality implies that $\|(DS_\bc(x,\ba)-DS_\bc(\tilde x,\bb))^*\eta\|\ge 3\alpha_0\|\sC\|^{q}\|\sE^{-1}\|^q\|\eta\|$. So, by the last inequality,
	$$
	2\|\sC^q\|\|\eta\|> \frac{27}{10}\|\sE^{-q}\|^{-1}\|\sC\|^q\|E^{-1}\|^q\|\eta\|.
	$$

	Since $\|\sE^{-q}\|\le \|\sE^{-1}\|^q$, it follows $\|\sE^{-q}\|\|\sE^{-1}\|^q\ge 1$, and therefore
	$
	2\|\sC^q\|\|\eta\|> \frac{27}{10}\|\sC\|^q\|\eta\|,
	$
	which is a contradiction.
\end{proof}

To make the local argument we will consider a fixed partition of  unity.  For this purpose,  consider $\{\chi_\bc:\torus \to \R\}_{\bc\in\mathcal{A}^p}$ a family of $C^\infty$ functions that form a partition of unity subordinated to the covering $\{\mathcal{R}_*(\bc)\}$.

We define  $\{\chi_{\bc,\ba}:\torus \to \R\}_{\bc\in\mathcal{A}^p}$ by
\begin{equation}
\chi_{\bc,\ba}(E^{-q}_{\bc,\ba} (x) ) = \chi_\bc(x)
\end{equation}
if $x\in\mathcal{R}_*(\bc)$ and $0$ elsewhere. Notice that $\{\chi_{\bc,\ba}\}$ is another partition of unity  subordinated to $\{\mathcal{R}_*(\bc\ba)\}$.

The following lemma compares the $H^s$ norm of $\phi$ with the sums of $H^s$ norm of $\chi_\bc\phi$, defined by  $\chi_\bc\phi(x,y):=\chi_\bc(x)\phi(x,y)$.

\begin{lemma}\label{partition.unity}
	There exists a constant $K$ such that, for any $\phi\in C^r(D)$, it holds
	\begin{equation}\label{partition.1}
	\sum_{(\ba,\bc)\in A^q\times A^p}\|\chi_{\bc,\ba}\phi\|^2_{H^s} \leq  2 \|\phi\|^2_{H^s} + K \|\phi\|^2_{L^1}
	\end{equation}
	and
	\begin{equation}\label{partition.2}
	\|\phi\|^2_{H^s} \leq K \sum_{\bc \in A^p} \|\chi_\bc \phi\|^2_{H^s} + K \|\phi\|^2_{L^1}.
	\end{equation}
\end{lemma}
\begin{proof} First consider the case $s\in \N$. Then
\begin{align*}
\sum_{(\ba,\bc)\in A^q\times A^p}\|\chi_{\bc,\ba}\phi\|_{H^s}^{2}&=\sum_{(\ba,\bc)\in A^q\times A^p}\sum_{|\sigma|\le s}\|\partial^{\sigma}(\chi_{\bc,\ba}\phi)\|_{L^2}^{2}   \le K(s)(I_1 +I_2), 
\end{align*}
where
$$
I_1=\sum_{(\ba,\bc)\in A^q\times A^p}\sum_{|\sigma|\le s} \int_{\T^u \times \R^d}|\chi_{\bc,\ba}(x,y)|^2|\partial^{\sigma}\phi(x,y)|^2dxdy
$$
and
$$
I_2=\sum_{(\ba,\bc)\in A^q\times A^p}\sum_{|\sigma|\leq s}\sum_{\sigma'<\sigma}{\sigma \choose \sigma'}\int_{\T^u \times \R^d}|\partial^{\sigma-\sigma'}\chi_{\bc,\ba}(x,y)|^2|\partial^{\sigma'}\phi(x,y)|^2\,dxdy.
$$

Since $\chi_{\bc,\ba}$ is a partition of unity,  $I_1$ is bounded by 
$$
\sum_{|\sigma|\le s} \int_{\T^u \times \R^d}|\partial^{\sigma}\phi(x,y)|^2dxdy=\|\phi\|_{H^s}^2
$$
and $I_2$ is bounded by $K(s,p,q)\|\phi\|_{H^{s-1}}^2$. 
Using Young's inequality, it follows \eqref{partition.1}.

For the case $s\not \in \N$, let $t$ be the largest integer that is less than $s$ and $\delta=s-t \in (0,1)$. Then
\begin{align*}
\sum_{\ba,\bc} \|\chi_{\bc,\ba}\phi\|^2_{H^s} &\le \sum_{\ba,\bc} \sum_{|\sigma|\le t} \|\partial^{\sigma} (\chi_{\bc,\ba}\phi) \|^2_{L^2} + \sum_{(\ba,\bc)\in A^q\times A^p}\sum_{|\sigma|= t} R(\ba,\bc,\sigma) = \mathcal{S}_1 + \mathcal{S}_2,
\end{align*}
where, considering $X=(x,u)$ and $V=(v,w)$,
$$
R(\ba,\bc,\sigma)=\int_{\T^u \times \R^d}\int_{\R^u \times \R^d}\frac{|\partial^{\sigma}(\chi_{\bc,\ba}\phi)(X+V)-\partial^{\sigma}(\chi_{\bc,\ba}\phi)(X)|^2}{(|V|^2)^{u+d+2\delta}}\,dXdV.
$$

As in the previous case, $\mathcal{S}_1$ is bounded by $\|\phi\|_{H^t}^2$. To estimate $\mathcal{S}_2$, let us write
$$
R(\ba,\bc,\sigma)=R_1(\ba,\bc,\sigma) + R_2(\ba,\bc,\sigma) + R_3(\ba,\bc,\sigma) + R_4(\ba,\bc,\sigma)
$$ 
with
$$
R_1=\int_{\T^u \times \R^d}\int_{\R^u \times \R^d}|\chi_{\bc,\ba}(X)|^2\frac{|\partial^{\sigma}\phi(X+W)-\partial^{\sigma}\phi(X)|^2}{(|W|^2)^{u+d+2\delta}}\,dXdW
$$

$$
R_2=\int_{\T^u \times \R^d}\int_{\R^u \times \R^d}\frac{|\chi_{\bc,\ba}(X+W)-\chi_{\bc,\ba}(X)|^2}{(|W|^2)^{u+d+2\delta}}|\partial^{\sigma}\phi(X+W)|^2\,dXdW
$$

$$
R_3=\sum_{\sigma'<\sigma}\int_{\T^u \times \R^d}\int_{\R^u \times \R^d}|\partial^{\sigma-\sigma'}\chi_{\bc,\ba}(X)|^2\frac{|\partial^{\sigma'}\phi(X+U)-\partial^{\sigma'}\phi(X)|^2}{(|W|^2)^{u+d+2\delta}}\,dXdW
$$

\begin{align*}
R_4=\sum_{\sigma'<\sigma}\int_{\T^u \times \R^d}\int_{\R^u \times \R^d}\frac{|\partial^{\sigma-\sigma'}\chi_{\bc,\ba}(X+U)-\partial^{\sigma-\sigma'}\chi_{\bc,\ba}(X)|^2}{(|W|^2)^{u+d+2\delta}}&|\partial^{\sigma'}\phi(X+W)|^2\,dXdW
\end{align*}

So, $\mathcal{S}_1+\sum_{(\ba,\bc)\in A^q\times A^p}\sum_{|\sigma|=t}R_1(\ba,\bc,\sigma)$ is bounded by $\|\phi\|_{H^s}^2$ and
$R_2+R_3+R_4$ is bounded by $K(s)\|\phi\|_{H^{t}}^2$. Using Young's inequality again, we have \eqref{partition.1}.

For inequality \eqref{partition.2}, note that the closure of each $\mathcal{R}_*(\bc)$ intersects at most $r$ closures of the sets $\overline{\mathcal{R}_*(\tilde{\bc})}$, since the Markov partition is formed by $r$ sets. Then:
\begin{align*}
\|\phi\|^2_{H^s} = \sum_{c,c'} \langle \chi_c \phi, \chi_{c'} \phi \rangle_{H^s} = \sum_{\cR_*(c)\cap\cR_*(c')=\emptyset} \langle \chi_c \phi, \chi_{c'} \phi \rangle_{H^s} + \sum_{\cR_*(c)\cap\cR_*(c')\neq \emptyset} \langle \chi_c \phi, \chi_{c'} \phi \rangle_{H^s}.
\end{align*}

If $\cR_*(c)\cap\cR_*(c')=\emptyset$, then Remark \ref{support} implies that
$\langle \chi_c\phi, \chi_{c'}\phi \rangle_{H^s} \leq K \| \chi_c\phi \|_{L^1} \|\chi_{c'}\phi\|_{L^1}$,
which gives:
$$\sum_{\cR_*(c)\cap\cR_*(c')=\emptyset} \langle \chi_c \phi, \chi_{c'} \phi \rangle_{H^s}\leq K \|\phi\|^2_{L^1}.$$

If $\cR_*(c)\cap\cR_*(c')\neq\emptyset$, then $\langle \chi_c\phi, \chi_{c'}\phi \rangle_{H^s}  \leq \frac{\|\chi_c\phi\|^2_{H^s} + \| \chi_{c'}\phi\|^2_{H^s}}{2}$. So, we have
$$
\sum_{\cR_*(c)\cap\cR_*(c')\neq \emptyset} \langle \chi_c \phi, \chi_{c'} \phi \rangle_{H^s}\le r\sum_{\bc \in A^p} \|\chi_\bc \phi\|^2_{H^s}.
$$

\end{proof}

\begin{lemma}\label{change.of.variables} Given $0\le s\le r $,  $\ba\in\mathcal{A}^q$ and $\bc\in\mathcal{A}^p$, there exists a constant $K>1$ such that
	\begin{equation}\label{change}
	\|\cL^q(\chi_{\bc\ba}\phi)\|^2_{H^s} \leq \frac{K }{(|\det E| |\det C| \fm(C)^{2s}))^{q}} \|\chi_{\bc,\ba}\phi\|^2_{H^s} 
	\end{equation}
for every $\phi \in C^r(D)$.
\end{lemma}
\begin{proof} Let us first consider $s$ integer.
	Recalling that $T_{\bc,\ba}^{-q}$ is an inverse branch defined over $\cR(\bc) \times \R^d$ by $T_{\bc,\ba}^{-q}(x,y)=(\sE_{\bc,\ba}^{-q}x,\sC^{-q}(y-S_{\bc}(x,\ba)))$. If we call by $g_1, g_2, \dots, g_{u+d}$ the components of $T^{-q}_{\bc,\ba}$ then we may observe that $|\partial^{\sigma}g_j\|\le \alpha_0 \|C^{-q}\|$ for all $\sigma$ multi-index with $|\sigma| \le r$.

Noticing that $\cL^q(\chi_{\bc,\ba}\phi) =\frac{(\chi_{\bc,\ba}\phi) \circ T_{\bc,\ba}^{-q}}{|\det DT|^q} $ and
	 recalling the formula for differentiation of the composition (Remark \ref{composition}), we have
\begin{align*}
	 	|\det DT|^{2q} \|\cL^q(\chi_{\bc,\ba}\phi)\|^2_{H^s} &= |\det DT|^{2q} \sum_{|\sigma|{\le}s} \int |\partial^{\sigma}\cL^q(\chi_{\bc,\ba}\phi)(z) |^2 dz \\
	 	&\leq \sum_{|\sigma|{\le}s} \int |\partial^{\sigma} [(\chi_{\bc,\ba}\phi)\circ T_{\bc,\ba}^{-q}](z)|^2 dz \\
		&\leq \sum_{|\sigma|{\le}s} \int |\sum_{|\sigma'|\le|\sigma|} (\partial^{\sigma'} \chi_{\bc,\ba}\phi)\circ T_{\bc,\ba}^{-q}(z) \psi_{\sigma',\sigma}(z)|^{2} dz
	 	\end{align*}

Above we used Remark~\ref{composition} to write $\partial^{\sigma}[(\chi_{\bc,\ba}\phi)\circ T_{\bc,\ba}^{-q}](z)$ as $
\sum_{|\sigma'|\le|\sigma|}  \psi_{\sigma',\sigma}(z) (\partial^{\sigma'} \chi_{\bc,\ba}\phi)\circ T_{\bc,\ba}^{-q}(z),
$
where $\psi_{\sigma',\sigma}$ is a polynomial function of degree at most $s$ in the variables $\partial^{\gamma}g_j$, where $\gamma$ goes through the multi-indexes with
$|\gamma|\leq |\sigma'|$ and $j=1,2,\dots,u+d$.

Noticing that $|\psi_{\sigma',\sigma}(z)|\le K\|\sC^{-q}\|^s$
for some constant $K$, 
we have

\begin{align}\label{sobolev.changeofvariables}
\nonumber	 	|\det DT|^{2q} \|\cL^q(\chi_{\bc,\ba}\phi)\|^2_{H^s} &\leq K^2\|\sC^{-q}\|^{2s} \sum_{|\sigma|{\le}s} \int \big(\sum_{|\sigma'|\le|\sigma|} |(\partial^{\sigma'} \chi_{\bc,\ba}\phi)\circ T_{\bc,\ba}^{-q}(z)|\big)^2 dz\\
\nonumber	 	&\leq K^2\|\sC^{-q}\|^{2s}|\det DT^q| \sum_{|\sigma|{\le}s} \int \big(\sum_{|\sigma'|\le|\sigma|} |(\partial^{\sigma'} \chi_{\bc,\ba}\phi)(z)|\big)^2 dz\\
\nonumber &\le K\|\sC^{-q}\|^{2s}|\det DT^q| \sum_{|\sigma|{\le}s} \int |(\partial^{\sigma} \chi_{\bc,\ba}\phi)(z)|^2 dz\\
\nonumber &= K\|\sC^{-q}\|^{2s}|\det DT^q|\|\chi_{\bc,\ba}\phi\|_{H^s}^2.
\end{align}

This implies \eqref{change} in this case. 
	 	
 		For non-integers values of $s$, we consider integers $s_0$ and $s_1$ with $0\leq s_0 \leq s \leq s_1 \leq r$. Since $\phi$ is $C^r(D)$, then it is in $H^{s_0}$ and $H^{s_1}$.
 		Applying Claim \ref{interpolacao} for $K_0 = \big( \frac{K}{(|\det \sE|\cdot|\det \sC| \fm(C)^{2s_0} )^q}\big)^{\frac{1}{2}} $ and $K_1= \big(\frac{K}{(|\det \sE|\cdot|\det \sC| \fm(C)^{2s_1} )^q} \big)^{\frac{1}{2}} $, it follows \eqref{change}.

\end{proof}

Now we can proceed to the Proof of Lemma \ref{l.norma.sobolev}.

\begin{proof}[Proof of Lemma \ref{l.norma.sobolev}]
ByLemma \ref{partition.unity}, we have
\begin{align*}
\| \cL^q \phi\|^2_{H^s} &\leq K \sum_{\bc \in A^p} \|\chi_\bc \cL^q \phi\|^2_{H^s} + K \|\phi\|^2_{L^1} \\
&\leq K \sum_{\bc \in A^p} \| \sum_{\ba\in A^q} \cL^q (\chi_{\bc\ba}\phi)  \|^2_{H^s} + K \|\phi\|^2_{L^1} \\
&= K \sum_{\bc\in A^p,\,\ba,\bb \in A^q} \langle \cL^q(\chi_{\bc\ba}\phi),\cL^q(\chi_{\bc\bb}\phi) \rangle_{H^s} + K \|\phi\|^2_{L^1}.
\end{align*}

In the following, we estimate $\langle \cL^q(\chi_{\bc\ba}\phi),\cL^q(\chi_{\bc\bb}\phi) \rangle_{H^s}$ dividing it into 2 cases: when $\ba \pitchfork_\bc \bb$ and when $\ba \not\pitchfork_\bc \bb$

 If $\ba \pitchfork_\bc \bb$, by Lemma \ref{transversality.lemma},  for every $(\xi,\eta) \neq (0,0)$ we have  either $(DT_x^q)^*(\xi,\eta)\in \cC_1^*$ for all $x\in \cR(\bc\ba)$ or $(DT^q_x)^*(\xi,\eta)\in \cC_1^*$ for all $x\in\cR(\bc\bb)$. Denote by $U$ the set of $(\xi,\eta)$ such that the first occurs and $V$ the set such that the second occurs.

Then, if $(\xi,\eta)\in U$, by Lemma \ref{estimar.transf.vs.l1} we have
\begin{align}
|\cF(\cL^q(\chi_{\bc,\ba}\phi))(\xi,\eta)|
&\leq K ( 1+ |\xi|^2 + |\eta|^2)^{-\frac{\rho_0}{2}} \|\phi\|^\dagger_{\rho_0}.
\end{align}

Remind that $\|\cL^q \phi\|_{H^s} \leq K^q \|\phi\|_{H^s}$ since the operator is bounded in $H^s$ (see Remark \ref{composition}).
So, by Cauchy-Schwarz

\begin{align*}
\Big| \sum_{\xi}\int_U &(1+|\xi|^2+|\eta|^2)^s \cF\cL^q(\chi_{\bc,\ba}\phi)(\xi,\eta) \overline{\cF\cL^q(\chi_{\bc,\bb}\phi)}(\xi,\eta) d\eta   \Big| \\
	&\leq \left( \sum_{\xi}\int_U (1+|\xi|^2+|\eta|^2)^s |\cF\cL^q(\chi_{\bc,\ba}\phi)(\xi,\eta)|^2  d\eta  \right)^{1/2} \|\cL^q(\chi_{\bc,\bb}\phi)\|_{H^s}\\
	&\leq K \sum_{\xi} \int (1+|\xi|^2+|\eta|^2)^{s-\rho_0} \|\phi\|^\dagger_{\rho_0}d\,\eta \leq K(q) \|\phi\|^\dagger_{\rho_0} \|\phi\|_{H^s},
\end{align*}
where we used that the integral is finite since $s-\rho_0<-(u+d)/2$.

Summing with the same integrals over $V$ instead of $U$, we obtain that
\begin{equation}\label{eq. transv}
\langle \cL^q(\chi_{\bc\ba}\phi),\cL^q(\chi_{\bc\bb}\phi) \rangle_{H^s} \leq  K(q) \|\phi\|^\dagger_{\rho_0} \|\phi\|_{H^s}
\end{equation}

If $\ba \not\pitchfork_\bc \bb$, we use
\begin{equation}
\langle \cL^q(\chi_{\bc\ba}\phi),\cL^q(\chi_{\bc\bb}\phi) \rangle_{H^s} \leq \frac{\|\cL^q(\chi_{\bc\ba}\phi)\|^2_{H^s}+ \|\cL^q(\chi_{\bc\bb}\phi)\|^2_{H^s} }{2}
\end{equation}
together the definition of $\tau(q)$ to obtain
\begin{equation}\label{eq. naotransversal}
\sum_{\ba \not\pitchfork_\bc \bb} \langle \cL^q(\chi_{\bc\ba}\phi),\cL^q(\chi_{\bc\bb}\phi) \rangle_{H^s} \leq \tau(q) \sum_{\ba} \|\cL^q(\chi_{\bc\ba}\phi)\|^2_{H^s}
\end{equation}

Using it,
Lemma \ref{partition.unity} and Lemma \ref{change.of.variables}, we have

\begin{align}
 \nonumber \sum_{\bc\in A^p,\,\ba\not\pitchfork_\bc \bb} \langle \cL^q(\chi_{\bc,\ba}\phi)&,\cL^q(\chi_{\bc,\bb}\phi) \rangle_{H^s} \leq \frac{K \tau(q)}{(|\det E| |\det C| \fm(C)^{2s})^q} \sum_{\ba,\bc} \| \chi_{\bc,\ba} \phi \|^2_{H^s} \\
 &\leq \frac{K \tau(q)}{(|\det E| |\det C| \fm(C)^{2s})^q}  \|  \phi \|^2_{H^s} + K(q) \|\phi\|_{L^1}^2. \label{ntransv}
\end{align}

Since $\|\cdot\|_{L^1} \leq \|\cdot \|^\dagger_{\rho}$ and $\|\cdot\|_{L^1} \leq K\|\cdot \|_{L^2} \leq K\|\cdot\|_{H^s}$, we may use \eqref{eq. transv} and \eqref{ntransv} to get the estimate below

\begin{align*}
\|\mathcal{L}^q \phi\|_{H^s}^2 &\leq \sum_{\bc\in A^p,\,\ba\not\pitchfork_\bc \bb} \langle \cL^q(\chi_{\bc\ba}\phi),\cL^q(\chi_{\bc\bb}\phi) \rangle_{H^s} + \sum_{\bc\in A^p,\,\ba\pitchfork_\bc \bb} \langle \cL^q(\chi_{\bc\ba}\phi),\cL^q(\chi_{\bc\bb}\phi) \rangle_{H^s} + K \|\phi\|_{L^1}\\
&\leq \frac{K \tau(q)}{(|\det E| |\det C| \fm(C)^{2s})^q}  \|  \phi \|^2_{H^s} +K \|\phi\|^2_{L^1}+    K(q) \|\phi\|^\dagger_{\rho_0} \|\phi\|_{H^s}  \\
&\leq  \frac{K \tau(q)}{(|\det E| |\det C| \fm(C)^{2s})^q}  \|  \phi \|^2_{H^s} + K(q)  \|\phi\|^\dagger_{\rho_0} \|\phi\|_{H^s}
\end{align*}
\end{proof}

\section{Proof of Theorems \ref{teo.1} and \ref{teo.2}}

Let us proceed to put together the two main Lasota-Yorke inequalities to obtain the third Lasota-Yorke inequality of this work, from which will follow Theorems 1 and 2.

\subsection{Third Lasota-Yorke (for $\|\phi\|=\|\phi\|_{H^s} + \|\phi\|^\dagger_{\rho_0}$)}

Putting the two Main Inequalities together, we obtain a third Lasota-Yorke.

\begin{proposition}[Third Lasota-Yorke]\label{togheter} Given $q\in \N$ satisfying $B_1 \frac{ \tau(q)}{(|\det DT| {\fm(C)^{2s}})^q} < 1$ and integers $0 \leq \rho_1<\rho_0 \leq r-1$ with $s<\rho_0-u-d$, consider $\nu=\nu(\rho_0,\rho_1) := \sum_{j=\rho_1+1}^{\rho_0}\frac{1}{j}  $ and some
	\begin{equation}\label{intervalo}
	\zeta \in \Big( \max\big\{ \|E^{-1}\|^{{\frac{1}{\nu}}} , \big(\{B_1^{\frac{1}{q}} \frac{\tau(q)^{1/q}}{|\det DT| {\fm(C)^{2s}}}\big)^{\frac{1}{2}} \big\}  , 1   \Big).
	\end{equation}
	
	Consider also the norm $\|\phi\|:=\|\phi\|_{H^s} + \|\phi\|^\dagger_{\rho_0}$, then there exists a constant $K$ such that for all $n\in \N$,
	\begin{equation}
	\|\cL^n \phi \| \leq K \zeta^n \|\phi\| + K \| \phi\|^\dagger_{\rho_1}.
	\end{equation}	
\end{proposition}

\begin{proof}[Proof of Proposition \ref{togheter}]

Let us begin this proof stating two  consequences of the first Main Inequality (Lemma \ref{l.norma cruz}). 

\begin{claim}\label{l. norma cruz 2}
  Let $\delta\in (\|E^{-1}\|,1)$. There exists $K>0$ such that, for $1\le \rho\le r-1$, for $n\in\N$,
  \begin{equation}
  \|\cL^nh\|_{\rho}^{\dagger}\le K\delta^{\rho n}\|h\|_{\rho}^{\dagger} + K\|h\|_{\rho-1}^{\dagger}.
\end{equation}
\end{claim}
\begin{proof} Take $N\in\N$ such that $K\|E^{-1}\|^{\rho N}\le \delta^{\rho N}$. Then,
  by Lemma \ref{l.norma cruz}, we have
  \begin{equation}\label{eq. norma cruz 2}
    \|\cL^Nh\|_\rho^{\dagger}\le \delta^{\rho N}\|h\|_\rho^{\dagger}+K(N)\|h\|_{\rho-1}^{\dagger}
  \end{equation}

Moreover, using $\|h\|_{\rho-1}^{\dagger}\le \|h\|_{\rho}^{\dagger}$, there exists $K=K(N)>0$ such that
\begin{equation}\label{eq. local1}
  \|\cL^jh\|_{\rho}^{\dagger}\le \delta^{rN}K\|h\|_{\rho}^{\dagger}\le \delta^{\rho j}K\|h\|_{\rho}^{\dagger}
\end{equation}
for all $1\le j \le N-1$ and $1\le\rho\le r-1$.

We prove by induction on $\rho$ that there exists a constant $K_\rho>0$ such that for every $n\in\N$
\begin{equation}
    \|\cL^nh\|_\rho^{\dagger}\le K_\rho\delta^{\rho n}\|h\|_\rho^{\dagger}+K_\rho\|h\|_{\rho-1}^{\dagger} \quad \text{ and } \quad 
  \|\cL^nh\|_{\rho}^{\dagger}\le K_\rho\|h\|_{\rho}^{\dagger}.
\end{equation}

Write $n=kN+j$, with $1\le j\le N-1$. For $\rho=1$ we have
$$
\begin{aligned}
    \|\cL^nh\|_1^{\dagger}=\|\cL^{kN}(\cL^jh)\|_1^{\dagger}&\le \delta^{kN}\|\cL^jh\|_1^{\dagger}+K(N)\sum_{i=0}^{k-1}\delta^{iN}\|\cL^{(k-1-i)N}(\cL^jh)\|_{0}^{\dagger}\\
                                                     &\le K\delta^{n}\|h\|_1^{\dagger}+K(N)\frac{K}{1-\delta^N}\|h\|_{0}^{\dagger}\\
                                                     &\le K_1\delta^n\|h\|_1^{\dagger}+K_1\|h\|_0^{\dagger}
\end{aligned}
$$
and
$$
\|\cL^nh\|_1^{\dagger}\le K_1\|h\|_1^{\dagger},
$$
where $K_1=2\max\{K,\ds\frac{K(N)A_0}{1-\delta^N}\}$.

Now, suppose the result is true for $\rho-1$, we prove for $\rho$.
$$
\begin{aligned}
    \|\cL^nh\|_\rho^{\dagger}=\|\cL^{kN}(\cL^jh)\|_\rho^{\dagger}&\le \delta^{kN}\|\cL^jh\|_\rho^{\dagger}+K(N)\sum_{i=0}^{N-1}\delta^{iN}\|\cL^{(k-1-i)N}(\cL^jh)\|_{\rho-1}^{\dagger}\\
                                                     &\le K\delta^{n}\|h\|_\rho^{\dagger}+K(N)\frac{K_{\rho-1}}{1-\delta^N}\|h\|_{\rho-1}^{\dagger}\\
                                                     &\le K_\rho\delta^n\|h\|_\rho^{\dagger}+K_\rho\|h\|_{\rho-1}^{\dagger}
\end{aligned}
$$
and
$$
\|\cL^nh\|_\rho^{\dagger}\le K_\rho\|h\|_\rho^{\dagger},
$$
where $K_\rho=2\max\{K,\ds\frac{K(N)K_{\rho-1}}{1-\delta^N}\}$. 
The result follows taking $K=\underset{1\leq i \leq r-1}{\max}\{K_i\}$.
\end{proof}

\begin{claim}\label{l. norma cruz 3}
	Given $\delta \in( \|E^{-1}\|,1)$ and integers $0\leq \rho_1 < \rho_0 \leq r-1$, let $\nu(\rho_0,\rho_1)$ be as before. Then there exists $K>0$ such that, for every $n \in \N$,
	\begin{equation}
	\|\cL^n h \|^\dagger_{\rho_0} \leq K \delta^{n/\nu(\rho_0,\rho_1)} \|h\|^\dagger_{\rho_0} + K \|h\|^\dagger_{\rho_1}.
	\end{equation}
\end{claim}
\begin{proof}
	Let $n$ be a multiple of $(r-1)!$, then we have by induction on $\rho \in [\rho_1 + 1, \rho_0]$ that
	\begin{equation}
	\|\cL^{n(\nu(\rho,\rho_1))} h \|^\dagger_\rho \leq K_{\rho} \delta^n \|h\|^\dagger_\rho + K_{\rho} \|h\|^\dagger_{\rho_1}.
	\end{equation}
	
	Actually, the case $\rho=\rho_1+1$ is immediately because $\nu(\rho_1 +1,\rho_1)=1/(\rho_1+1)$. Also, using Claim \ref{l. norma cruz 2}, the relation $n\nu(\rho+1,\rho_1)=n\nu(\rho,\rho_1)+\frac{n}{\rho+1}$ and the induction hypothesis, we have:
	\begin{align*}
		\|\cL^{n(\nu(\rho+1,\rho_1))} h \|^\dagger_{\rho+1} &= \|\cL^{\frac{n}{\rho+1}}(\cL^{n(\nu(\rho,\rho_1))} h ) \|^\dagger_{\rho+1}\\
		&\leq K \delta^n \|\cL^{n(\nu(\rho,\rho_1))} h\|^\dagger_{\rho+1} + K \|\cL^{n(\nu(\rho,\rho_1))} h\|^\dagger_{\rho}\\
		&\leq K \delta^n \big(K \delta^{n\nu(\rho,\rho_1) (\rho+1)}\|h\|_{\rho+1} + K \|h\|^\dagger_{\rho}\big) + \big(K_{\rho}\delta^n \|h\|^\dagger_{\rho}+ K_{\rho} \|h\|^\dagger_{\rho_1}\big)\\
		&\leq K_{{\rho +1}} \delta^n \|h\|^\dagger_{\rho +1} + K_{{\rho +1}} \|h\|^\dagger_{\rho_1}.
	\end{align*}
	
	So we have the lemma for multiples of $(r-1)!\nu(\rho_0,\rho_1)$. For the general case, just notice that Claim \ref{l. norma cruz 2} also implies that $\cL$ is a bounded operator with respect to the norm $\|\cdot\|^\dagger_\rho$.
\end{proof}

Now we proceed to prove Lemma \ref{togheter}, noticing first that
	for $a,b>0$, we have that $\sqrt{a+b}\leq \sqrt{a} + \sqrt{b}$ and $\sqrt{ab} \leq \epsilon a + \epsilon^{-1}b$. So {Lemma~\ref{l.norma.sobolev}} implies that for every $\epsilon>0$
	\begin{equation*}
	\|\cL^q \phi \|_{H^s} \leq \Big( K \frac{\tau(q)^{1/q}}{|\det DT|{\fm(C)^{2s}} }\Big)^{q/2} \|\phi\|_{H^s} + \epsilon \|\phi\|_{H^s} + K(\epsilon)\|\phi\|^\dagger_{\rho_0}.
	\end{equation*}

Since $(K^{1/q} \frac{\tau(q)^{1/q}}{|\det DT| {\fm(C)^{2s}}}  )^{q/2} < \zeta^q$, for $\epsilon=\epsilon(q)$ small we have
    \begin{equation*}
	\|\cL^q \phi \|_{H^s} \leq \zeta^q \|\phi\|_{H^s}  + K(q) \|\phi\|^\dagger_{\rho_0}.
	\end{equation*}
	
	Iterating it $l$ times:
    \begin{equation}\label{ws}
	\|\cL^{lq} \phi \|_{H^s} \leq \zeta^{lq} \|\phi\|_{H^s}  + K(l) \|\phi\|^\dagger_{\rho_0}.
	\end{equation}
	
	Now,
	taking $\delta$ slightly {smaller} than $\zeta^{\nu}$ and $\l_0$ large enough such that $K (\delta^{\frac{1}{\nu}})^{l_0q}<\frac{\zeta^{l_0q}}{2}$, Claim \ref{l. norma cruz 3} implies for $l_0$
	\begin{equation}\label{rho_0}
	\|\cL^{l_0q} \phi \|_{\rho_0} \leq \frac{\zeta^{l_0q}}{2} \|\phi\|_{\rho_0}  + K(l_0) \|\phi\|^\dagger_{\rho_1}.
	\end{equation}
	
	Let us consider the auxiliary norm $\|\phi\|^* := \|\phi\|_{H^s} + 2K(l_0)\zeta^{-l_0q} \|\phi\|^\dagger_{\rho_0}$, which is equivalent to $\|\cdot\|$. Adding ($\ref{ws}$) and ($\ref{rho_0}$), it follows that:
	\begin{equation}
	\|\cL^{l_0q} \phi \|^* \leq \zeta^{l_0q} \|\phi\|^*  + \tilde{K}(l_0) \|\phi\|^\dagger_{\rho_1}.
	\end{equation}
	
	Iterating this inequality, it follows what we want for every $n$ but for the norm $\|\cdot\|^*$. Since they are equivalent norms, it follows the result for the norm $\|\cdot\|$.
\end{proof}

%

\subsection{Proof of Theorem \ref{teo.1}}

\begin{proof}[Proof of Theorem \ref{teo.1}]
Since $|\det DT| {\fm(C)}^{2s}>1$, the transversality condition implies that we can consider $q$ such that
$\ds\omega= \frac{B_1 \tau(q)}{(|\det DT|{\fm(C)}^{2s})^q}<1$.
	
Consider $\rho_0 = r-1$ and $\rho_1=0$. Since $s<r-u/2-d/2-1$, we have that $s+u/2+d/2<\rho_0 $, so we can apply Lemma  \ref{togheter} for some $\zeta$ between $\omega$ and $1$.

Let us fix some non-negative function $\psi_0 \in C^r(D)$  with $\|\psi_0\|_{L^1}=1$, $\nu_0=\psi_0 m$,   $\psi_n = \frac{1}{n} (\psi_0 + \cL\psi_0 + \cdots + \cL^{n-1}\psi_0)$ and $\nu_n=\psi_n m$. Then 
$\nu_n =  \frac{1}{n} \sum_{j=0}^{n-1}T^j_*\nu_0$.

Since $\mu$ is the SRV measure for $T$, for every $\phi \in C^0(D)$ we have that  $\frac{1}{n}\sum_{j=0}^{n-1}\phi\circ T^j(x)$ converges to $\int \phi \, d\mu$	for Lebesgue almost every $x$, therefore 
\begin{equation}\label{eq. density_Hs1}
  \int \phi d\nu_n=\int \frac{1}{n}\sum_{j=0}^{n-1}\phi\circ T^j\,d\nu_0 \to \int \phi \, d\mu.
\end{equation}

	On the other hand, Lemma~\ref{togheter} implies that there exists a constant $K>0$, such that $\|\cL^n\psi_0\|\le K\|\psi_0\|$, for all $n$. In particular, $\|\psi_n\|_{H^s}\le\|\psi_n\|\le K\|\psi_0\|$ for every $n$. So,  Banach-Alaoglu theorem implies that there is a subsequence $\{\psi_{n_k}\}_k$ which converges weakly to some function $\psi_{\infty} \in H^s$, then
\begin{equation}\label{eq. density_Hs2}
  \int \phi d\nu_{n_k}=\int \phi\,\psi_{n_k}\,dm \to \int \phi\,\psi_{\infty} \, dm
\end{equation}
for every $\phi\in C^{r}(D)$ with compact support. Hence $\mu=\psi_{\infty}m$ is an absolutely continuous invariant probability.

The openness in $(C,f)$  follows from the fact that  $\tau(q)$ is upper semi-continuous on $(C,f)\in \mathcal{C}(d) \times C^r(\torus,\R^d)$ and from the openness of the condition $B_{1}\frac{\tau(q)}{|\det E||\det C|\fm(C)^{2s}}<1$,
what concludes the proof of the theorem.

%
%
\end{proof}

\begin{remark}\label{openness}
	It is important to mention that the transversality condition defined in Definition \ref{transversality.1} is not an open condition. What is open in $(C,f)$ is the condition $$\displaystyle \frac{B_1  \tau_f(q)}{|\det DT^q| \fm(C)^{2sq}} <1$$ for fixed $q$. 
\end{remark}

\subsection{Spectral Gap}

When $s>u/2$, we can apply a theorem of Hennion to obtain spectral properties of the action of the operator $\cL$ in a Banach space $\mathcal{B}$ contained in $H^s$ and containing $C^{r-1}(D)$. 

	Let us denote the spectral radius of $\cL:\mathcal{B}\to \mathcal{B}$ by
	$\rho(\cL)=  \lim_{n\to \infty} \sqrt[n]{\|\cL^n\|}$.
	We say that $\cL$ has spectral gap if there exist bounded operators $\cP$ and $\cN$ such that $\cL=\lambda \cP +\cN$, with $\cP^2=\cP$,  $\dim(\operatorname{im}(\cP))=1$, $\rho(\cN)<|\lambda|$ and $\cP\cN=\cN\cP=0$. 

The spectral gap can be obtained as a standard consequence of a Theorem due to Hennion, Ionescu Tulcea-Marinescu, et al \cite{hennion, tulcea.marinescu}:

\begin{theorem*}[Hennion]\label{hennion}
	Let $L:(\mathcal{B},\|\cdot\|) \to (\mathcal{B},\|\cdot\|)$ be a bounded operator and $\|\cdot\|'$ be a norm in $\mathcal{B}$ such that
	\begin{enumerate}
		\item $\|\cdot\|'$ is continuous in $\|\cdot\|$.
		
		\item For every bounded sequence $\{\phi_n\}\in \mathcal{B}$, there exists a subsequence $\{\phi_{n_k}\}$ and $\psi \in \mathcal{B}$ such that $\|\phi_{n_k} - \psi \|'\to 0$.
		
		\item $\|L \phi\|' \leq M \| \phi\|'$ for some $M>0$ and every $\phi\in \mathcal{B}$.
		
		\item There exists $r \in(0,\rho(L))$ and $K>0$ such that for all $n\in \N$
		\begin{equation}\label{doeblin.fortet}
		\|L^n \phi\| \leq r^n \|\phi\| + K \|\phi\|'.
		\end{equation}
		\item There exists a unique eigenvalue $\lambda$ with $|\lambda|=\rho(L)$ and $\dim\operatorname{ker}(L-\lambda I)=1$.
	\end{enumerate}

		Then $L$ has spectral gap.
	\end{theorem*}

Inequality (\ref{doeblin.fortet}) is sometimes known as Lasota-Yorke \cite{lasota.yorke} or Doeblin-Fortet \cite{doeblin.fortet} inequality for $L$ with respect to the spaces $\mathcal{B}$ and $\mathcal{B}'$. It is exactly the same kind of inequality that we proved in Lemmas \ref{l.norma cruz}, \ref{l.norma.sobolev} and \ref{togheter}.

\begin{definition}\label{decay.of.correlations} We say that
$(T,\mu)$ has \textbf{exponential decay of correlations} in a vector space $\mathcal{B} \subset L^1(\mu)$ with exponential rate at most $\zeta<1$ if for every $\phi \in \mathcal{B}$ and $\psi \in L^\infty(\mu)$, there exists a constant $K(\phi,\psi)>0$ such that
\begin{equation}\label{decaimento}
\left|\int \phi  (\psi\circ T^n) d\mu  - \int \phi d\mu \int \psi d\mu \right| \leq K(\phi,\psi) \zeta^n.
\end{equation}

\end{definition}

When the transfer operator $\cL$ has spectral gap, it follows that the dynamics has exponential decay of correlations with exponential rate at most $\rho(\mathcal{N})<1$, as given in the following Proposition.

\begin{proposition}\label{spectralgap.implies.decay}
Supposing that $\cL$ has spectral gap in some Banach space $\mathcal{B}$ embedded continuously in $L^2(m)$ 
	 with $\rho(\cL |\mathcal{B})=1$ and $\rho(\cN)=\zeta<1$,
	if we consider $\phi_0\in \mathcal{B}$  a  nonnegative fixed point of $\cL$ satisfying $\int \phi_0 \,dm=1$ and $\mu=\phi_0 m$, then $(T,\mu)$ has exponential decay of correlations in $\mathcal{\tilde B}:=\{\phi\in \mathcal{B},  \phi\phi_0\in \mathcal{B}\}$ with exponential rate at most $\zeta$. In particular, if $\mathcal{B}$ is a Banach algebra then $(T,\mu)$ has exponential decay of correlations in $\mathcal{B}$.
\end{proposition}
\begin{proof}
Since $\cL$ has spectral gap in $\mathcal{B}$, for each $\phi\in \mathcal{B}$ we  write $\phi=a(\phi)\phi_0+ \phi_1$ with   $\|\mathcal{L}^n\phi_1\|_{\mathcal{B}}\le \zeta^n\|\phi_1\|_\mathcal{B}$. Then the property $\int \cL u \, dm=\int u \, dm$ and $\cL  \phi_0 = \phi_0$ implies that $\int \phi_1 dm=0$  and $a(\phi)=\int \phi dm$.
We also have
$\cL^n(\phi\cdot\psi\circ T^n)=\psi\cdot\mathcal{L}^n\phi$
  and $\|\phi_1\|_{\mathcal{B}} \leq K \|\phi\|_{\mathcal{B}}$.




Given $\psi \in L^\infty(\mu)$  and $\phi\in\mathcal{\tilde B}$, it follows that:


\begin{align*}
\Big| \int \phi(\psi\circ T^n)\,d\mu - \int \phi d\mu \int \psi d\mu \Big| &=\Big| \int \phi(\psi\circ T^n)\phi_0\,dm - \int \phi\phi_0 dm \int \psi\phi_0 dm \Big| \\
&=  \Big| \int \big[\cL^n (\phi\phi_0) - \Big(\int \phi \phi_0\,dm\Big) \phi_0\big] \psi dm \Big|\\
&=  \Big| \int \cL^n \Big[ (\phi\phi_0)_1 \Big] \psi dm \Big|\\
&\le \|\psi\|_{L^2(m)} \| \cL^n (\phi\phi_0)_1 \|_{L^2(m)}  \\
&\le  K \|\psi\|_{L^2(m)} \| \cL^n (\phi\phi_0) \|_{\mathcal{B} } \\
&\le  K  \|\psi\|_{L^2(m)} \| (\phi\phi_0)_1 \|_{\mathcal{B} }  \zeta^n = K(\phi,\psi) \zeta^n.
\end{align*}

%

So 
$(T,\mu)$ has exponential decay of correlations in  $\mathcal{\tilde B}$.
\end{proof}


These facts will be used to prove Theorem \ref{teo.2}.

\begin{proof}[Proof of Theorem \ref{teo.2}]
Consider the smallest integer $\rho_0$ and the greatest integer $\rho_1$ such that
$$
\rho_1+ u/2  < s < \rho_0 - u/2-d/2. 
$$	
Consider $t \in ( \rho_1+u/2,s)$ and an integer $q$ such that $B_1 \tau(q) < (|\det DT|{\fm(C)}^{2s})^q$.
Since $\rho_0 - \rho_1 \leq u + \frac{d}{2} +2 $ and $\sum_{j=1}^{n}1/j\le 1+\log(n-1)$, we have $\nu\le  1+\log(u+\frac{d}{2}+1) := a$.
	
%

 So, if $\zeta \in \Big( \max\{\|E^{-1}\|^{\frac{1}{a}},(\frac{(B_1 \tau(q))^{1/q}}{|\det DT| \fm(C)^{2s}})\} , 1 \Big)$ then $\zeta$ is in the interval in \eqref{intervalo}. 

	We will verify that the conditions of Theorem \ref{hennion} are satisfied considering $\mathcal{B}$ the completion of $C^r(D)$ with respect to the norm $\|\cdot\| = \|\cdot\|_{H^s} + \|\cdot\|^\dagger_{\rho_0}$ and $\mathcal{B}'$ the completion of $C^r(D)$ with respect to the norm $\|\cdot\|^\dagger_{\rho_1}$.
	
	Obviously $\|\cdot\|_{\rho_1} \leq \|\cdot\|_{\rho_0} \leq \|\cdot\|$, which implies condition (1) in the theorem of Hennion. Condition (3) is and immediate consequence of Lemma \ref{l.norma cruz} and condition (4) follows from Lemma \ref{togheter} with $r=\zeta$. Condition (5) is immediate since $T$ is mixing in $\Lambda$.
	
	It remains to verify the compactness (condition (2)). 	
	The embedding of $H^s(D)$ in $H^t(D)$ is compact, by Sobolev's embedding theorem ($s>t$). So, it is sufficient to prove that the embedding of $H^t(D)$ in $\mathcal{B}'$ is {continuous}, which will be proved in Lemma \ref{continuous}.

%
%
%

	Finally, we notice that $C^{r-1}(D) \subset \mathcal{B}$. The definition of $\|\cdot\|^\dagger_\rho$ gives that
	$$\Big| \int_{U_\psi} \phi(y) \partial^\alpha_x \partial_y^\beta h(\psi(y),y) dy\Big| \leq \int_{U_\psi} \Big|\partial_x^\alpha \partial_y^\beta h (\psi(y),y)\Big|dy \leq
	 \operatorname{vol}(D) \|h\|_{C^{r-1}}$$
	whenever $|\alpha|+|\beta| \leq \rho_0$, $\psi \in \mathcal{S}$, $\phi \in C^{|\alpha|+|\beta|}(U_\psi)$ and $\|\phi\|_{C^{|\alpha|+|\beta|}}\leq 1$. This implies immediately that $\|h\|_{\rho_0}^\dagger \leq K \|h\|_{C^{r-1}}$ and that 	
	 $C^{r-1}(D) \subset \mathcal{B}$.   
%
\end{proof}

\begin{lemma}\label{continuous}
Consider $0 \leq \rho_1<\rho_0 \leq r-1$ such that
$$\rho_1+ u/2 < t  < s < \rho_0 - u/2-d/2.$$
	
Then the embedding of $H^t(D)$ in $\mathcal{B}'$ is continuous, that is, there exists a constant $K>0$ such that
$$\|u\|_{\rho_1}^{\dagger} \le K \|u\|_{H^t}.$$
\end{lemma}
\begin{proof}[Proof of Lemma \ref{continuous}]
Consider  $\psi=(\psi_1,\dots,\psi_{u})\in \mathcal{S}$ a $C^r$ transformation as in the Subsection \ref{ss. set S}. Recall that $\|D^\nu \psi\|\le k_\nu$ for $k_1, k_2, \dots, k_r$ previously fixed. Considering an extension of $\psi$, we may suppose that the domain $U_\psi$ of $\psi$ contains $\pi_2(D)=[-K_0,K_0]^{d}$. In these conditions, we establish

\begin{claim}\label{l. embedding}
Let $u:\T^u \times \R^d \to \R$ be a $C^r$ function with compact support in $D$. Define $v(x,y)=u(x+\psi(y),y)$ for $y\in U_\psi$ and $v(x,y)=0$ if $y\not \in U_{\psi}$. Then,
for every multi-index $\gamma$ with $|\gamma| \le r$ and for every $y\in U_\psi$, we have
\begin{align}
\label{eq. embedding 10} (\partial^{\gamma}u)(x+\psi(y),y)&=\sum_{|\tilde\gamma|\le |\gamma|}a_{\tilde\gamma,\gamma}(y)(\partial^{\tilde\gamma}v)(x,y)\,\,  \quad  \quad \quad  \text{and} \\
\label{eq. embedding 11}  (\partial^{\gamma}v)(x,y)&=\sum_{|\tilde\gamma|\le |\gamma|}b_{\tilde\gamma,\gamma}(y)(\partial^{\tilde\gamma}u)(x+\psi(y),y),
\end{align}
where $a_{\tilde\gamma,\gamma}$ and $b_{\tilde\gamma,\gamma}$ are polynomials of degree at most $|\gamma|$ in the variables $\partial^{\beta}\psi_k$, with $1\le|\beta|\le |\gamma|$. 
Consequently $\partial^{\alpha}a_{\tilde\gamma,\gamma}$ and $\partial^{\alpha}b_{\tilde\gamma,\gamma}$ are bounded by some constant $K$ which depends on only $k_1, k_2, \dots, k_r$, for all multi-index $\alpha=(\alpha_1, \dots, \alpha_d)$, with $0\le|\alpha|\le r-|\gamma|$.
\end{claim}
\begin{proof}
  Follows by induction on $\rho=|\gamma|$.
\end{proof}

\begin{claim}\label{l. embedding 1}
Let $u:\T^u \times \R^d \to \R$ be a $C^r$ function with compact support in $D$. Define $v(x,y)=u(x+\psi(y),y)$ for $y\in U_\psi$ and $v(x,y)=0$ if $y\not \in U_{\psi}$. Then, for any $0< t < r$, we have
$$
\|v\|_{H^t(\R^{u+d})}\le K \|u\|_{H^t(\R^{u+d})},
$$
where $K$ depends on only $k_1, k_2, \dots, k_r$.
\end{claim}
\begin{proof}
Follows using inequality \eqref{eq. embedding 11}.
\end{proof}

From the definition of $\|.\|_{\rho_1}^{\dagger}$ and Cauchy-Schwarz inequality, we have that 
\begin{equation}\label{eq. embedding 22}
\|u\|_{\rho_1}^\dagger \leq \operatorname{vol}(D) \max_{|\gamma|\le \rho_1}\sup_{\psi\in \cS}\|\partial^{\gamma}u(\psi(.),.)\|_{L^2(\R^d)}.
\end{equation}

By Claim~\ref{l. embedding}, the right-hand side of \eqref{eq. embedding 22} is bounded by
\begin{equation}\label{eq. embedding 23}
\operatorname{vol}(D) K\sum_{|\gamma|\le \rho_1}\|\partial^{\gamma}v(0,.)\|_{L^2(\R^d)}=K\|v(0,.)\|_{H^{\rho_1}}.
\end{equation}

Due to \cite[Theorem 7.58(iii)]{adams} applied with $p=q=2$, $\tilde{s}=\rho_1+u/2$, $\chi=\rho_1$, $k=d$, $n=u+d$, we have:
$$\|v(0,.)\|_{H^{\rho_1}} \leq K \|v(.,.)\|_{H^{\tilde{s}}(\R^{u+d})}.$$

By $t>\rho_1 + \frac{u}{2} = \tilde{s}$ and Claim~\ref{l. embedding 1}, we conclude that
\begin{equation}\label{eq. embedding 24}
 \|v(.,.)\|_{H^{\tilde{s}}(\R^{u+d})} \le  K \|v(.,.)\|_{H^t(\R^{u+d})}\le K \|u\|_{H^t(\R^{u+d})}.
\end{equation}

Therefore we have that $\|u\|_{\rho_1}^{\dagger} \le K \|u\|_{H^t}.$

\end{proof}

\section{Genericity}

In \cite[Theorem 2.12]{bocker.bortolotti}  the authors proved that if $C\in \mathcal{C}(d)$ satisfies $\| \sC \| < \frac{\|\sE^{-1}\|^{-1}}{|\det \sE|^{\frac{1}{{u-d+1}}}}$, then there 
 exists a family $f_\bt$, $\bt \in \R^m$, with $f_{0}=f$ such that the set $\big\{\bt \in \R^m , \underset{q\to\infty}{\limsup} \frac{1}{q}\log$ $\tau_{f_\bt}(q)  > \log J \big\}$ has zero Lebesgue measure (where $J=|\det E | |\det C|^{-1} \|C^{-1}\|^{-2d}>1$).

Actually, the same same proof is valid considering any $\beta>0$ instead of $\log J$, that is, if we define
\begin{equation}
\displaystyle \mathcal{C}(d;\sE)=\left\{ \sC \in \mathcal{C}(d) , \|\sC\| < \frac{\|\sE^{-1}\|^{-1}}{|\det \sE|^{\frac{1}{{u-d+1}}} } \right\},
\end{equation}
then 
the set 
$ \mathcal{T}_{\beta} :=
\{\bt \in \R^m , \underset{q\to\infty}{\limsup} \frac{1}{q}\log \tau_{f_\bt}(q)  > \beta  \}
$
has zero Lebesgue measure.

\begin{proposition}[\cite{bocker.bortolotti}, Theorem 2.12]\label{bb}
	 Given $\beta>0$, integers $u\ge d \ge 1$, $E\in \mathcal{E}(u)$ and $\sC\in \mathcal{C}(d,E)$, there exist $C^\infty$-functions $\phi_k:\T^u \to \R^d$, $k=1,2,\dots, m$ such that for $f_0\in C^2(\T^u,\R^d)$ and its corresponding family $f_\bt = f_0 + \sum_{k=1}^s t_k \phi_k$, the set of parameters $t=(t_1,t_2,\dots,t_m)$ such that $\bt \notin \mathcal{T}_\beta$ has full Lebesgue measure.
\end{proposition}	

As a consequence, for every $n\geq 1$ there exists a residual set $\mathcal{R}_n \in C^r(\torus, \R^d)$ such that $\underset{q\to +\infty}{\limsup} \frac{1}{q} \log \tau_f(q) < \frac{1}{n}$. Then $\mathcal{R}=\underset{n \geq 1}{\cap}\mathcal{R}_n$ is also a residual subset of $C^r(\torus, \R^d)$ such that $\underset{q\to +\infty}{\limsup} \frac{1}{q} \log \tau_f(q)=0$ for every $f \in \mathcal{R}$.

\subsection{Proof of the Main  Theorems}

Putting Proposition \ref{bb} together with Theorems \ref{teo.1} and \ref{teo.2}, it follows Theorems A and D, and the immediate Corollaries B and C.

\begin{proof}[Proof of Theorem A]
Consider $\beta = \log |\det \sC| \det \sE| \|C^{-1}\|^{-2s}   >0$, $\mathcal{R}$ as given by the consequence of Proposition \ref{bb} and $\mathcal{V} \subset C^r(\torus, \R^d)$ the set of $f$'s such that  the corresponding SRB measure $\mu_T$ of $T(C,E,f)$ is absolutely continuous with repect to the Lebesgue measure and
 $\|d\mu_T/d\operatorname{vol}_{\torus\times \R^d}\|_{H^s}<+\infty$.

As Theorem \ref{teo.1} is valid for every $f \in \mathcal{R}$, we have a corresponding open set $\mathcal{U}_f$ such that the conclusion of Theorem \ref{teo.1} is valid for every $g \in  \mathcal{U}_f$.
Taking $\mathcal{U}=\underset{f \in \mathcal{R}}{\cup} \mathcal{U}_f$, it follows that
 $\mathcal{R}\subset{\mathcal U}$ and that $\mathcal{U}$ is dense. So
 $\mathcal{U}$ is open and dense and Theorem A is valid for every $f \in \mathcal{U}$.
	
\end{proof}

\begin{proof}[Proof of Corollary \ref{cor.b}]
	Notice that the inequality $\displaystyle |\det \sE||\det \sC| \|C^{-1}\|^{-2s} > 1$ 
	is open in $s \in \R$. So if it is valid for $s=0$, then is also valid for some $s>0$.
\end{proof}

\begin{proof}[Proof of Corollary \ref{cor.c}]
	This corollary is immediate from Theorem A and Sobolev's embedding Theorem (the elements of $H^s(\torus \times \R^d)$ are continuous up to a null Lebesgue set when $s> \frac{u+d}{2}$).
\end{proof}

\begin{proof}[Proof of Theorem D]
Consider the same residual set $\mathcal{R}\subset C^r(\torus, \R^d)$ as in the proof of Theorem A and $\mathcal{U}$ the set of $f$'s such that the conclusion of Theorem \ref{teo.2} is valid.

As Theorem \ref{teo.2} is valid for every $f \in \mathcal{R}$, it follows that $\mathcal{U}$ is open and dense.
\end{proof}

\end{document}